\documentclass[12pt]{article}

 \usepackage{amsmath,amssymb,amscd,amsthm,esint}
 
\usepackage{graphics,amsmath,amssymb,amsthm,mathrsfs}

\usepackage{graphics,amsmath,amssymb,amsthm,mathrsfs,amsfonts,accents}
\usepackage{appendix}

\usepackage{sidecap}
\usepackage{float}
\usepackage{extarrows}
\usepackage{booktabs}
\usepackage{verbatim}
\usepackage{hyperref}
\usepackage[usenames,dvipsnames]{xcolor}

\newtheorem{thm}{Theorem}[section]
\newtheorem{lemma}[thm]{Lemma}
\newtheorem{cor}[thm]{Corollary}

\theoremstyle{definition}
\newtheorem{remark}[thm]{Remark}

\def\XXint#1#2#3{{\setbox0=\hbox{$#1{#2#3}{\int}$}
         \vcenter{\hbox{$#2#3$}}\kern-.5\wd0}}

\def\R{\mathbb{R}}
\def\C{\mathbb{C}}
\def\e{\varepsilon}

\def\Rd{\mathbb{R}^{d}}

\def\Cd{\mathbb{C}^d}
\def\Cdd{\mathbb{C}^{d\times d}}

\def\hW{\accentset{\circ}{W}}

\def\loc{{\rm loc}}

\def\H{\mathbb{H}}

\numberwithin{equation}{section}

\begin{document}

\title{Resolvent Estimates for the Stokes Operator \\ in  a Three-Dimensional Lipschitz Domain}

\author{
Zhongwei Shen 
}
\date{}

\maketitle

\begin{abstract}

By refining  the approach  developed in \cite{AGH-2015, GS2026a}, we establish resolvent estimates in $L^p_\sigma (\Omega)$
 for the Stokes operator in 
a bounded Lipschitz domain $\Omega$ in $\R^3$ for $(3/2)-\e< p\le \infty$, where $\e>0$ depends on $\Omega$.
As a consequence,  the Stokes operator generates a uniformly bounded analytic semigroup
in $L^p_\sigma (\Omega)$.
The range for $p$ is sharp.
The results are particularly surprising, as it is long believed that the $L^p$ resolvent estimates in three-dimensional 
Lipschitz domains may fail for large  $p$.

In the appendices  we prove an interpolation theorem between $L^2_\sigma(\Omega)$ and
$L^\infty_\sigma(\Omega)$. We also show that for
$1< p< \infty$ and a bounded Lipschitz domain $\Omega$ in $\R^d, d\ge 2$, 
 the resolvent estimate in $L^p(\Omega; \C^d)$ can hold only if 
 the Helmholtz projection is bounded on this space.
This explains the counter-example constructed in \cite{Deuring-2001}.	
Finally, we  prove that the resolvent estimate may fail in $L^p_\sigma (\Omega)$ for $1< p< (3/2)$ and $d\ge 3$
as a consequence of failure of the Helmholtz decomposition of $L^p(\Omega; \C^d)$.

\medskip


\medskip

\noindent {\it MR (2020) Subject Classification}: 35Q30.

\end{abstract}


\section{Introduction}

Let $\Omega$ be a  Lipschitz domain in $\R^d$.
Consider the resolvent  problem for the  Stokes operator with the Dirichlet condition,
\begin{equation}\label{eq-0}
\left\{
\aligned
-\Delta  u +\nabla \phi  +\lambda u & = F  & \quad & \text{ in } \Omega, \\
\text{\rm div}(u) & =0 & \quad  & \text{ in } \Omega,\\
u&=0 & \quad & \text{ on } \partial\Omega,
\endaligned
\right.
\end{equation}
where $\lambda\in \Sigma_\theta $ is a  (given) parameter and 
\begin{equation}\label{Sigma}
\Sigma_\theta
=\left\{
z\in \mathbb{C}\setminus \{ 0\}: \ |\text{arg} (z)|< \pi -\theta
\right\}
\end{equation}
for $\theta \in (0, \pi/2)$.
For $1< p\le \infty$,
let 
\begin{equation}\label{L-sigma}
L^p_\sigma (\Omega)
=\left\{ F \in L^p(\Omega; \C^d):\
\text{\rm div} (F) =0 \text{ in } \Omega \text{ and } F \cdot n=0 \quad \text{ on } \partial\Omega \right\},
\end{equation}
where $n$ denotes the outward unit normal to $\partial\Omega$. 
A weak formulation of $L^p_\sigma (\Omega)$ is given by \eqref{L-p}.

The following  is  the main result of the paper.

\begin{thm}\label{main-1}
Let $\Omega$ be a bounded Lipchitz domain   in $\R^3$.
Let  $\lambda \in \Sigma_\theta$, where $\theta \in (0, \pi/2)$.
Then for any $F\in L_\sigma^\infty(\Omega)$,  the weak solution of the Dirichlet problem \eqref{eq-0}   in 
$W_0^{1, 2}(\Omega; \C^3) \times L^2(\Omega; \C)$
satisfies the estimate,
\begin{equation}\label{est-0}
 |\lambda| \| u \|_{L^p (\Omega)}
\le C  \| F \|_{L^p (\Omega)}
\end{equation}
 for  $(3/2)-\e< p\le \infty$, where $\e>0$ depends on $\Omega$.
 The range for $p$ is sharp.
 The constant  $C$  in \eqref{est-0}   depends only on $\theta$, $p$ and $\Omega$.
 Moreover, if $\Omega$ is a bounded convex domain in  $\R^3$, the resolvent estimate \eqref{est-0}
 holds for $1< p\le   \infty$.
\end{thm}

Let 
$$
C_{0, \sigma}(\Omega)=
\left\{ F\in C(\overline{\Omega}; \C^d): \ \text{\rm div}(F)=0 \text{ in } \Omega
\text{ and } F=0 \text{ on } \partial\Omega\right \}.
$$

\begin{cor}
Let $\Omega$ be a bounded Lipschitz domain in $\R^3$.
Then the Stokes operator generates a uniformly bounded analytic semigroup  of angle $\pi/2$
on $C_{0, \sigma}(\Omega)$ and on $L^p_\sigma (\Omega)$ for $(3/2) -\e < p< \infty$.
 It also generates a (non-$C_0$)
 uniformly bounded analytic semigroup on $L^\infty_\sigma (\Omega)$ of angle $\pi/2$.
 If $\Omega$ is a bounded convex domain in  $\R^3$, the corresponding results hold for
 $1< p\le \infty$.
\end{cor}

If $\Omega$ is a bounded or exterior domain with $C^{1, 1}$ boundary, the $L^p$ resolvent estimate \eqref{est-0},
together with the corresponding $L^p$ estimates for $\nabla u$, $\nabla^2u$ and $\nabla\phi $,
holds for  $1< p< \infty$; see \cite{Solo-1977, Giga-1981, Sohr-1987, Sohr-1994} and their references.
 The endpoint case $p=\infty$,  which had remained open for many years, 
 was  solved  for $C^2$ domains  in a series of papers \cite{AG-2013,AG-2014, AGH-2015}.
 
 We now turn to  the case of nonsmooth domains, which are the main  concern of this paper. 
If $\Omega$ is a bounded Lipschitz domain in $\R^d$, it was proved in \cite{Shen-2012}  by the present author  that
the resolvent estimate \eqref{est-0}
 holds whenever  $d\ge 3$ and
\begin{equation}\label{Lip-range}
\Big| \frac{1}{p}-\frac12 \Big|< \frac{1}{2d} +\e,
\end{equation}
where $\e>0$ depends on the Lipschitz character of $\Omega$.
In particular, when $d=3$, the estimate \eqref{est-0} holds for 
 $(3/2)-\e< p< 3+\e$.
Using a similar approach, F. Gabel and P. Tolksdorf  \cite{Tolksdorf-2022}  established \eqref{est-0}
for $d=2$ and $(4/3)-\e< p< 4+\e$.
In  \cite{GS-2023}  J. Geng and the  present author proved  that if $\Omega$ is a bounded or exterior $C^1$ domain,
then the resolvent estimate
\begin{equation}\label{est-p}
|\lambda| \| u\|_{L^p(\Omega)}
+|\lambda|^{1/2} \| \nabla u \|_{L^p (\Omega)}
\le C \| F \|_{L^p(\Omega)}
\end{equation}
 holds
for  $1< p< \infty$. 
More recently \cite{GS2026a}, by adapting the method developed by K. Abe, Y. Giga, and M. Hieber  in \cite{AGH-2015}, we 
establish the resolvent estimate \eqref{est-0} for $p=\infty$ when either 
$d\ge 3$ and $\Omega$ is $C^1$ or $d=2$ and $\Omega$ is Lipschitz.
As a result, 
 the  estimate \eqref{est-0} holds for any $1< p\le \infty$ 
on bounded $C^1$ domains in $\R^d$, $d\ge 3$, and for $(4/3)-\e < p \le \infty$
on bounded Lipschitz domains in $\R^2$
 \footnote{It was stated in \cite[Remark 1.3]{GS2026a} that the two dimensional result holds for $1< p\le \infty$. 
 This is incorrect.
 The duality argument, which relies on the boundedness of the Helmholtz projection,
 only yields  the range $(4/3)-\e< p < 2$, which is already known \cite{Tolksdorf-2022}. 
In fact, the range $(4/3)-\e < p\le \infty$ for $d=2$  is sharp. See Theorem \ref{thm-c2}.
}.
Our Theorem \ref{main-1}  shows  that  the same estimate holds for every $(3/2)-\e < p\le \infty$,
 when $\Omega$ is a bounded Lipschitz domain 
in $\R^3$. This result is particularly unexpected in view of the longstanding belief that the $L^p$ resolvent estimate \eqref{est-0} may fail for large $p$ in 
three-dimensional Lipschitz domains; see
 \cite{Deuring-2001}.
For $d\ge 4$, the range of $p$'s  given by \eqref{Lip-range} in \cite{Shen-2012}  remains  the best known result for Lipschitz domains.
 The question of the sharp range for \eqref{est-0} remains open for $d\ge 4$.
 
 Theorem \ref{main-1} for the endpoint case $p=\infty$  is proved by refining the approach developed in \cite{AGH-2015, GS2026a}.
 Our starting point is an $L^q$ generalized resolvent estimate \eqref{G-est} in a localized Lipschitz domain $E_r$
 of size $r$, defined by \eqref{E}, 
 for some $q> \frac{2d}{d-1}$. This is obtained by combining the $L^p$ resolvent theorem
  in \cite{Shen-2012} for $p$ satisfying \eqref{Lip-range}, 
 with  $W^{1,p}$ estimates for the stationary Stokes equations \cite{BS-1995, Mitrea-2004, MW-2012}, corresponding
  to the case $\lambda=0$, in a Lipschitz domain.
  Next,   note that if $d=3$, the $L^q$ generalized resolvent estimate \eqref{G-est} holds  for some $q>3$.
 Following \cite{AGH-2015, GS2026a}, we apply the Sobolev inequality to obtain 
 \begin{equation}\label{i-1}
 |\lambda| \| u \|_{L^\infty(E_r)}
 \le C r^{-\frac{3}{q}}
 \left\{ |\lambda| \| u\varphi \|_{L^q(E_{Nr})}
 +|\lambda|^{1/2} \|\nabla (u\varphi ) \|_{L^q(E_{Nr})} \right\},
 \end{equation}
 where  $r=|\lambda|^{-1/2}$ and $N\ge 4$ is a large parameter to be chosen. Here, 
 $\varphi$ is a cut-off function satisfying 
 $$
 \varphi=1 \quad \text{  in } E_r, \quad  \|\nabla  \varphi \|_\infty \le C (Nr)^{-1}, \quad 
 \|\nabla^2 \varphi \|_\infty \le C (Nr)^{-2}.
 $$
 We then apply  the generalized $L^q$ resolvent estimate mentioned above to the localized problem
 in order to bound the right-hand side of \eqref{i-1}.
 
 In \cite{GS2026a}, the boundary terms arising from this localization were controlled by
  the regularity estimates for the  stationary Stokes equations and harmonic functions
  in $C^1$ domains. 
  More precisely, those estimates provide control of
 $$
 \| |\nabla u| +|\phi| \|_{L^q(\Delta_{Nr})}
 $$
 for $q>d$ (see \eqref{E} below for the definition of $\Delta_r$).
 Such estimates are unavailable in three-dimensional Lipschitz domains, even for Laplace's equation.
 We overcome this difficulty by using a trace theorem and nontangential-maximal-function
 estimates, which reduce the boundary integrals to 
$$
\| | \nabla u| +|\phi|\|_{L^s(\Delta_{Nr})},
$$
 where $s=\frac{q (d-1)}{d}$.
In dimension  $d=3$, taking  $q=3+\e$ gives  $s=2+(2\e/3)$.
The resulting boundary integrals can therefore be controlled 
using the classical results  of Fabes, Kenig, and Verchota \cite{FKV-1988} for Lipschitz domains.
This is the main new ingredient in the proof of Theorem \ref{main-1} for the case $p=\infty$.



Theorem \ref{main-1} for  $(3/2)-\e < p< \infty$ follows from the endpoint case $p=\infty$ and the trivial case $p=2$
by interpolation and duality. 
For the range $2< p< \infty$, 
since we have been unable to locate  in the literature a proof for  the required interpolation result between 
$L^2_\sigma(\Omega)$ and $L^\infty_\sigma (\Omega)$, even for smooth domains, 
we include one in Appendix A.
The interpolation theorem, which is of independent interest,
 is established  for any bounded Lipschitz domain $\Omega$ in $\R^d$, $d\ge 2$
(the two dimensional case is needed for a result in \cite{GS2026a}).
 Its proof uses a potential representation in \cite{CM}
  to preserve the divergence-free condition in the decomposition of
 vector fields. The argument can be adapted to prove
 the real
interpolation identity
\begin{equation}\label{B-1}
\left[ L_\sigma^2 (\Omega), L^\infty_\sigma (\Omega) \right]_{\theta, p}
=L^p_\sigma (\Omega)
\end{equation}
for $2< p< \infty$, where $\theta =1-\frac{2}{p}$.
We note that the standard approach based on  the Helmholtz projection  does not apply to $L^\infty_\sigma (\Omega)$,  even when $\Omega$ is  smooth.
For Lipschitz domains, it   may also fail when  $p>3$ and $d\ge 3$ \cite{Fabes1998}.
The relationship between the Helmholtz decomposition and the analyticity of the Stokes semigroup 
is an interesting and subtle question.
In \cite{Giga2017} it was shown that  for a section-like smooth domain $\Omega$  in $\R^2$, 
the Stokes operator generates an analytic semigroup in $L^p_\sigma (\Omega)$ for $p>2$, even though 
the $L^p$-Helmholtz decomposition may  fail. 
Our result provides analogous  examples among bounded Lipschitz domains in $\R^3$.
In Appendix B, we show that if  $1< p< \infty$ and 
the estimate \eqref{est-0}  holds for any $f\in C_0^\infty(\Omega; \C^d)$, then 
the Helmholtz projection must be bounded on $L^p(\Omega; \C^d)$. 
This explains the counter-example constructed in \cite{Deuring-2001}, since
for any $p>3$, there exists a bounded Lipschitz domain $\Omega$ in $\R^3$ for which the projection is unbounded 
on $L^p(\Omega; \C^d)$ \cite{Fabes1998}.

For  $p< 2$, we use a duality argument that  relies on the boundedness of the Helmholtz projection.
It was proved in \cite{Fabes1998} that 
for a bounded Lipschitz domain in $\R^d$, the projection is bounded on $L^p(\Omega; \C^d)$ for $(3/2)-\e < p< 3+\e$
if $d\ge 3$, where $\e>0$ depends on $\Omega$ (if $d=2$, the corresponding range is $(4/3)-\e < p< 4+\e$).
This gives the range $(3/2)-\e < p < 2$, which is already known \cite{Shen-2012}.
If $\Omega$ is convex, the $L^p$-Helmholtz decomposition holds for $1< p< \infty$ \cite{GS2010}. Consequently,
 the resolvent estimate \eqref{est-0} holds for $1< p\le \infty$
 if $\Omega$ is a bounded convex domain in $\R^3$ (or $\R^2$). 
 In  Appendix C, we show that if the uniqueness for the Helmholtz decomposition
fails for some $\frac{2d}{d+2}\le q< 2$, then  the resolvent estimate also fails for $F\in L^p_\sigma(\Omega)$
for $1< p\le q$. 
As a result, for bounded Lipschitz domains,
 the resolvent estimate in $L^p_\sigma(\Omega)$ may fail in general 
for $1< p< (3/2)$ if $d\ge 3$ and for $1< p< (4/3)$ if $d=2$. See Theorem \ref{thm-c2}.

We end this section with a few notations. 
Let $\psi: \mathbb{R}^{d-1} \to \mathbb{R}$ be a Lipschitz function with  $\psi (0)=0$ and $\|\nabla \psi \|_\infty \le M$, and 
\begin{equation}\label{H}
\H_{\psi}= \left\{ (x^\prime, x_d)\in \R^d: \ x^\prime \in \R^{d-1} \text{ and } x_d> \psi (x^\prime)\right\}.
\end{equation}
 Define
\begin{equation}\label{E}
\aligned
E_r & = \left\{ (x^\prime, x_d) \in \Rd: \  |x^\prime|< r  \text{ and } \psi (x^\prime) <x_d < C_0 r  \right\},\\
\Delta_r &= \left\{ (x^\prime,  x_d )  \in \R^d: \ |x^\prime|< r \text{ and } x_d =\psi(x^\prime)  \right\}
\endaligned
\end{equation}
for $0< r< \infty$.
There exist $C_0\ge M+1$ and $c_0\in (0, 1)$, depending only on $d$ and $M$,  such that  $E_r$ is star-like with respect to every
point in the ball centered at $(0, (1/2) C_0 r)$ with radius $c_0 r$ \cite{JK-1982}.
For now on, we shall fix the constants $M$ and $C_0$.
The star-like property allows us to use  a classical result of M.E. Bogovski\u{\i}.
Let 
$$
D_r =  B(0, r)\cap \H_\psi  \quad \text{ and } \quad I_r = B(0, r) \cap \H_\psi.
$$
Since 
$$
\aligned
D_{cr}    \subset E_r \subset D_{Cr} \quad \text{ and } \quad 
I_{cr}  \subset \Delta_r \subset I_{Cr},\\
\endaligned
$$
where $C>c>0$ depend only on $d$ and $M$, we may interchange $(E_r, \Delta_r)$ with $(D_r, I_r)$ in boundary estimates.

\medskip

\noindent{\bf Acknowledgement.}
The author is grateful to A. Gaudin and Y. Giga for their helpful comments on an earlier version of the manuscript,
 particularly regarding the relevance (and irrelevance)
of the $L^p$-Helmholtz decomposition to the $L^p$ resolvent estimate and analyticity of the Stokes operator.



\section{Generalized resolvent estimates}\label{section-L}

Let $0< r< \infty$ and $E_r$ be given by \eqref{E}.
Consider   the generalized  resolvent problem for the Stokes equations,
\begin{equation}\label{G-eq}
\left\{
\aligned
-\Delta u +\nabla \phi +\lambda u  & =F +\text{\rm div} (f) & \quad & \text{ in } E_r,\\
\text{\rm div}(u) & = g & \quad & \text{ in } E_r,\\
u & =0 & \quad  & \text{ on } \partial E_r,
\endaligned
\right.
\end{equation}
 where $\lambda\in \Sigma_\theta$.
 For $1< q< \infty$,  define
\begin{equation}\label{W}
\aligned
\hW^{1, q}(E_r; \C)=\left\{  u\in W^{1, q}   ({E_r}; \C): \  \int_{E_r} u =0  \right\},
\endaligned
\end{equation}
with the norm $\|\nabla u \|_{L^q(E_r )}$.
Let  $\hW^{-1, q}(E_r ; \C)$ denote  the dual of $\hW^{1, q^\prime} (E_r;  \C)$, where  $q^\prime =\frac{q}{q-1}$.

\begin{thm}\label{thm-G}
Let  $d\ge 3$ and $ \lambda\in \Sigma_\theta$.
There exists $\e \in (0, 1)$,  depending only on $d$,  $\theta$ and $M$, such that 
if  
\begin{equation}\label{q}
\Big| \frac{1}{q} -\frac12 \Big| < \frac{1}{2d} +\e,
\end{equation}
then 
for any $F\in L^q(E_r; \Cd)$, $f\in L^q(E_r ; \Cdd)$ and $g \in L^q(E_r; \C)$ with $\int_{E_r} g =0$,
there exists  a unique  $u \in W_0^{1, q}(E_r; \Cd)$
such that  \eqref{G-eq} holds for some $\phi\in L^q(E_r; \C)$.
Moreover, the solution satisfies 
\begin{equation}\label{G-est}
\aligned
 & |\lambda|^{1/2} \| \nabla u \|_{L^q(E_r)}
+|\lambda| \| u \|_{L^q(E_r)}
\\
&\quad \le C \left\{
\| F \|_{L^q(E_r)}
+ |\lambda|^{1/2}  \| f \|_{L^q(E_r )}
+|\lambda|^{1/2}  \| g \|_{L^q(E_r)}
+  |\lambda| \| g \|_{\hW^{-1, q}(E_r )}
\right\},
\endaligned
\end{equation}
where $C$ depends only on $d$, $q$, $\theta$ and $M$.
\end{thm}

The proof of Theorem \ref{thm-G} relies on the resolvent estimates in \cite{Shen-2012}
and uses the following  $W^{1, q}$ estimates for the stationary Stokes equations.

\begin{lemma}\label{local-w}
Suppose  $d\ge 2$ and    $q>1$ satisfies the condition \eqref{q}.
Let $(u, \phi )\in W_0^{1, 2} (E_1; \C^d) \times L^2(E_1; \C)$ be a weak solution of
\begin{equation}\label{G-0}
\left\{
\aligned
-\Delta u +\nabla \phi  & =F +\text{\rm div} (f) & \quad & \text{ in } E_1,\\
\text{\rm div}(u) & = 0 & \quad & \text{ in } E_1,\\
u & =0 & \quad  & \text{ on } \partial E_1.
\endaligned
\right.
\end{equation}
 Then
\begin{equation}\label{local-w1}
\| \nabla u \|_{L^q(E_1)} + \| u \|_{L^q(E_1)}
\le C \left\{ \| F \|_{L^q(E_1)} + \| f\|_{L^q(E_1)} \right\}.
\end{equation}
Moreover,  for any  $0<\rho <1/2$, 
\begin{equation}\label{local-w2}
\aligned
\left(\fint_{E_\rho } |\nabla u|^q \right)^{1/q}
 & \le \frac{C}{\rho}
\left(\fint_{E_{2\rho} } |u|^q\right)^{1/q}
+ C \rho  \left(\fint_{E_{2\rho} } |F|^q\right)^{1/q}
+ C \left(\fint_{E_{2\rho} } |f|^q\right)^{1/q}.
\endaligned
\end{equation}
The constant  $C>0$ depends only on $d$, $q$ and $M$.
\end{lemma}

\begin{proof}

The estimate \eqref{local-w1} for $d=3$ was proved in \cite{BS-1995},
while the remaining case  was covered in \cite{Mitrea-2004, MW-2012}.
The estimate \eqref{local-w2} follows from \eqref{local-w1}
by a standard localization procedure.
We point out that if the function $\psi$ in \eqref{E} is $C^1$, the 
estimate \eqref{local-w2} hold for any $q>1$ \cite{Galdi1994, Mitrea-2004}.
\end{proof}

\begin{proof}[\bf Proof of Theorem \ref{thm-G}]

We divide the proof into four steps.

Step 1. Reduction to the case where $g=0$, using a classical result of M.E. Bogovski\u{\i}.

Since $g\in L^q(E_r; \C)$ and $\int_{E_r} g=0$, it follows by Lemma B.2 in \cite{GS2026a} that there exists
$w\in W^{1, q}_0(E_r; \C)$ such that $\text{\rm div}(w)=g$ in $E_r$, 
\begin{equation}\label{w-1}
\|\nabla w \|_{L^q(E_r)} \le C \| g \|_{L^q(E_r)}
\quad \text{ and } \quad \| w \|_{L^q(E_r)}\le 
C \|  g \|_{\hW^{-1, q}(E_r)},
\end{equation}
where $C$ depends only on $d$, $q$ and $M$.
Let $v=u-w$. Then $v=0$ on $\partial E_r$ and
$$
\left\{
\aligned
-\Delta v +\nabla p +\lambda v & = F -\lambda w + \text{\rm div}(f) +\Delta w\\
\text{\rm div} (v) &=0
\endaligned
\right.
$$
in $E_r$. If the theorem holds for the case $g=0$, then
$$
\aligned
 & |\lambda|^{1/2} \| \nabla v \|_{L^q(E_r)}
+|\lambda| \| v \|_{L^q(E_r)}
\\
&\quad \le C \left\{
\| F \|_{L^q(E_r)} + |\lambda| \| w\|_{L^q(E_r)} 
+ |\lambda|^{1/2}  \| f \|_{L^q(E_r )}
+|\lambda|^{1/2}   \|\nabla w\|_{L^q(E_r)}
\right\}.
\endaligned
$$
It follows   that 
$$
\aligned
 & |\lambda|^{1/2} \| \nabla u \|_{L^q(E_r)}
+|\lambda| \| u \|_{L^q(E_r)}
\\
&\quad \le C \left\{
\| F \|_{L^q(E_r)} + |\lambda| \| w\|_{L^q(E_r)} 
+ |\lambda|^{1/2}  \| f \|_{L^q(E_r )}
+|\lambda|^{1/2}   \|\nabla w\|_{L^q(E_r)}
\right\}\\
&\quad
\le C \left\{
\| F \|_{L^q(E_r)}
+ |\lambda|^{1/2}  \| f \|_{L^q(E_r )}
+|\lambda|^{1/2}   \| g \|_{L^q(E_r)}
+\lambda| \| g \|_{\hW^{-1, q}(E_r)} 
\right\},
\endaligned
$$
where we have used \eqref{w-1} for the last inequality.

\medskip

Step 2. Reduction to the case $r=1$ and $g=0$ by rescaling.

Let  $(u, \phi)$ be a weak solution of \eqref{G-eq}  in $E_r$ with $g=0$.
Let $\widetilde{u}(x)= u(rx)$ and $\widetilde{\phi}(x) = r \phi (rx)$. 
Then
$$
\left\{
\aligned
-\Delta \widetilde{u} +\nabla \widetilde{\phi} +\widetilde{\lambda} \widetilde{ u}  & =\widetilde{F} +\text{\rm div} (\widetilde{f}) & \quad & \text{ in } \widetilde{E},\\
\text{\rm div}(\widetilde{u}) & = 0 & \quad & \text{ in } \widetilde{E},\\
\widetilde{u} & =0 & \quad  & \text{ on } \partial \widetilde{E},
\endaligned
\right.
$$
where $\widetilde{\lambda}=r^2 \lambda$, $\widetilde{F}(x)= r^2 F(rx)$, $\widetilde{f}(x)=  r f(rx)$, and
$$
\widetilde{E}=\big\{ (x^\prime, x_d): |x^\prime|< 1 \text{ and } \widetilde{\psi}(x^\prime )< x_d< C_0\big\}
$$
with $\widetilde{\psi}(x^\prime)=r^{-1} \psi (rx^\prime)$.
Since $\|\nabla \widetilde{\psi}\|_\infty \le M$,  the resolvent estimate for $\widetilde{u}$ in the case $r=1$ gives the 
estimate for $u$ in $E_r$. 

\medskip

Step 3. Resolvent estimates for the case $g=0$, $r=1$ and $f=0$.

Let $(u, \phi)$ be  a weak  solution of \eqref{G-eq} with $g=0$, $r=1$ and $f=0$.
Under the condition \eqref{q}, it was proved  in \cite{Shen-2012} that 
\begin{equation}\label{r-1}
(|\lambda|+1)  \| u \|_{L^q(E_1)}
\le C \| F \|_{L^q(E_1)},
\end{equation}
where $C$ depends only on $d$, $\theta$, $q$ and $M$.
To bound $|\lambda|^{1/2} \| \nabla u \|_{L^q(E_1)}$, we use Lemma \ref{local-w}
with $\rho =|\lambda|^{-1/2}$.
Suppose $0< \rho\le  c_0$. 
It follows by \eqref{G-0}  as well as well-known interior estimates for the stationary Stokes equations that 
\begin{equation}\label{r-2a}
\left(\fint_{E_1\cap B(x_0, \rho)} |\nabla u|^q \right)^{1/q}
\le \frac{C}{\rho}
\left(\fint_{E_1\cap B(x_0, 2\rho)} |u|^q\right)^{1/q}
+ C \rho
\left(\fint_{E_1\cap B(x_0, 2\rho)}
|F-\lambda u|^q \right)^{1/q}
\end{equation}
for $x_0\in E_1$. By using a simple covering argument and \eqref{r-1}, this leads to
$$
|\lambda|^{1/2} \| \nabla u \|_{L^q(E_1)}
\le C \| F \|_{L^q(E_1)}.
$$
If $|\lambda|^{-1/2}\ge c_0$, we may use \eqref{r-2a} with $\rho=c_0$ and \eqref{r-1} to show that 
$$
\| \nabla u \|_{L^q(E_1)} \le C \| F \|_{L^q(E_1)}.
$$
It follows that for any $\lambda\in \Sigma_\theta$,
\begin{equation}\label{r-2}
(|\lambda|+1) \| u \|_{L^q(E_1)}
+( |\lambda|+1)^{1/2} \|\nabla u \|_{L^q(E_1)}
\le C \| F \|_{L^q(E_1)}
\end{equation}
if $q>1$ satisfies the condition \eqref{q}.

\medskip

Step 4. Resolvent estimates for the case $g=0$, $r=1$ and $F=0$.

We use a duality argument to complete the proof.
Let $(u, \phi )$ be a weak solution of \eqref{G-eq} with $g=0$, $F=0$ and $r=1$.
For $G\in C_0^\infty(E_1; \C^d)$, let $(v, \xi)$ be a weak solution of 
$$
\left\{
\aligned
-\Delta v +\nabla \xi +  {\lambda} v & = G & \quad & \text{ in } E_1,\\
\text{\rm div} (v) & =0 & \quad & \text{ in } E_1,\\
v & =0 & \quad & \text{ on } \partial E_1.
\endaligned
\right.
$$
Since $q^\prime$ also satisfies the condition \eqref{q}, it follows from 
  \eqref{r-2} that 
$$
 (|\lambda|+1)^{1/2} \| \nabla v \|_{L^{q^\prime}(E_1)}
\le C \| G \|_{L^{q^\prime}(E_1)}.
$$
By the definition of weak solutions, 
$$
\int_{E_1} u \cdot G =\int_{E_1} 
\nabla u\cdot \nabla v + \lambda \int_{E_1} u \cdot v =
-\int_{E_1} f \cdot \nabla v.
$$
Hence,
$$
\aligned
\Big| \int_{E_1} u \cdot G\Big|
 & \le \| f \|_{L^q(E_1)} \| \nabla v \|_{L^{q^\prime}(E_1)}\\
& \le C  (|\lambda| +1)^{-1/2}   \| f \|_{L^q(E_1)}
\| G \|_{L^{q^\prime}(E_1)}.
\endaligned
$$
By duality we obtain 
\begin{equation}\label{r-3}
\| u \|_{L^q(E_1)}
\le C (|\lambda| +1)^{-1/2}   \| f \|_{L^q(E_1)}.
\end{equation}
By applying the local estimate \eqref{local-w}, as in Step 3, we see that 
$$
\|\nabla u \|_{L^q(E_1)} \le C \| f \|_{L^q(E_1)},
$$
which, together with \eqref{r-3}, gives the desired estimate \eqref{G-est} in the case where 
$r=1$, $g=0$ and $F=0$.
This completes the proof.
\end{proof}


\section{Localization}

For $r>0$, let $E_r$ and $\Delta_r$ be defined by \eqref{E}.

\begin{thm}\label{thm-L}
Suppose $d\ge 3$ and $q> 2 $ satisfies  the condition \eqref{q}.
Let $ \lambda\in \Sigma_\theta$
and  $(u, \phi)\in W^{1, q} (E_{Nr}; \C^d) \times L^q(E_{Nr}; \C)$ be a weak solution of
\begin{equation}\label{L-11}
\left\{
\aligned
-\Delta u +\nabla \phi +\lambda u & = F &\quad & \text{ in } E_{Nr},\\
\text{\rm div} (u) & =0 & \quad & \text{ in } E_{Nr},\\
u& =0 & \quad & \text{ on } \Delta_{Nr},
\endaligned
\right.
\end{equation}
where $N\ge 4$ and $F\in L^\infty(E_{Nr}; \C^d)$. 
Then 
\begin{equation}\label{L-12}
\aligned
&  |\lambda| \| u \|_{L^q(E_r)} 
+|\lambda|^{\frac12} \| \nabla u \|_{L^q(E_r)}\\
  & \le 
C \Bigg\{
\| F \|_{L^\infty(E_{Nr})}  (Nr)^{\frac{d}{q}}
+ \left(\fint_{E_{Nr}} |u|^q\right)^{1/q}
\left(
(Nr)^{\frac{d}{q}-2} + (Nr)^{\frac{d}{q}-1} |\lambda|^{\frac12} \right)\\
 & \quad + \left(\fint_{E_{Nr}} ( |\nabla u| + |\phi|)^q \right)^{1/q} (Nr)^{\frac{d}{q}-1}
+\left(\fint_{\Delta_{Nr}} ( |\nabla u| + |\phi|)^p \right)^{1/p}
(Nr)^{\frac{d-1}{p}-1}
\Bigg\},
\endaligned
\end{equation}
where  $p=\frac{q(d-1)}{ d}$ and $C$ depends only on $d$, $q$, $\theta$ and $M$ (not on $r$ or $N$).
\end{thm}

\begin{proof}
The proof is almost identical to that of Theorem 3.1 in \cite{GS2026a}, except for one step.
Let $Q_r =(-r, r)^{d-1} \times (-C_0 r, C_0 r)$. 
Choose  a cut-off function $\varphi \in C_0^\infty (Q_{Nr} )$ such that 
$\varphi =1$ in $Q_r$,
\begin{equation}\label{L-13a}
\|\nabla \varphi \|_\infty \le C (Nr)^{-1} \quad \text{ and } \quad
\|\nabla^2 \varphi \|_\infty \le C (Nr)^{-2},
\end{equation}
where $C$ depends only on $d$. Note that 
\begin{equation}\label{L-13}
\left\{
\aligned
-\Delta (u \varphi) + \nabla (\phi  \varphi) + \lambda (u \varphi) & = h & \quad & \text{ in } E_{Nr},\\
\text{\rm div} (u \varphi) & = g & \quad & \text{ in } E_{Nr},\\
u \varphi & =0  & \quad & \text{ on } \partial E_{Nr},
\endaligned
\right.
\end{equation}
where
\begin{equation}\label{L-14}
\left\{
\aligned
h & =F \varphi - 2 (\nabla u) (\nabla \varphi) - u \Delta \varphi + \phi  (\nabla \varphi),\\
g & =u \cdot \nabla \varphi.
\endaligned
\right.
\end{equation}
It follows by Theorem \ref{thm-G} that
\begin{equation}\label{L-15}
\aligned
   |\lambda| \| u  & \|_{L^q(E_r)} 
+|\lambda|^{\frac12} \| \nabla u \|_{L^q(E_r)}\\
  & \le 
C \left\{ \| h \|_{L^q(E_{Nr})}
+ |\lambda|^{\frac12} \| g \|_{L^q(E_{Nr})}
+ |\lambda|  \| g \|_{\hW^{-1, q}(E_{Nr})} \right\}.
\endaligned
\end{equation}
It is not hard to see that 
$\| g \|_{L^q(E_{Nr})} \le C  \| u \|_{L^q(E_{Nr})} (Nr)^{-1}$,
 and 
$$
\aligned
\| h \|_{L^q(E_{Nr})}
& \le C \Big\{
\| F \|_{L^\infty (E_{Nr})}  (Nr)^{\frac{d}{q}}
+ \|\nabla u \|_{L^q(E_{Nr})} (Nr)^{-1}\\
& \qquad\qquad
 + \| \phi  \|_{L^q(E_{Nr})} (Nr)^{-1} 
+ \| u \|_{L^q(E_{Nr})} (Nr)^{-2} \Big\}.
\endaligned
$$

To bound the term  $|\lambda| \| g \|_{\hW^{-1, q}(E_{Nr} )}$, 
 let $\eta \in \hW^{1, q^\prime}(E_{Nr} )$ with $\| \nabla \eta \|_{L^{q^\prime} (E_{Nr})}\le 1$. 
  For any $\beta \in \C$, we have 
$$
\aligned
\lambda \int_{E_{Nr}}  g   \eta
&=\lambda \int_{E_{Nr} } \text{\rm div} (u \varphi)   (\eta-\beta)
=\lambda \int_{E_{Nr}} (u \cdot \nabla \varphi) (\eta-\beta)\\
&=\int_{E_{Nr}}
\{ (F+\Delta u -\nabla \phi )\cdot \nabla \varphi  \} (\eta -\beta)\\
&=J_1 +J_2,
\endaligned
$$
where we have used \eqref{L-11}.
Let $\beta =\fint_{E_{Nr}} \eta$ .
As in \cite{GS2026a}, 
$$
\aligned
|J_1|
 & =\big|\int_{E_{Nr}} (F \cdot \nabla \varphi) ( \eta-\beta)\big|\\
 & \le C \| F \|_{L^\infty (E_{Nr})}  (Nr)^{\frac{d}{q}},
\endaligned
$$
and 
$$
\aligned
J_2 & =\int_{E_{Nr}} (  (\Delta u-\nabla \phi ) \cdot \nabla\varphi) (\eta-\beta)\\
&=-\int_{E_{Nr}} (\partial_j u^i -\delta_{ij} \phi  ) (\partial^2_{ij}  \varphi )  (\eta-\beta)
-\int_{E_{Nr}} (\partial_j  u^i -\delta_{ij} \phi  ) (\partial_i  \varphi )(\partial_j \eta)\\
&\qquad\qquad\qquad
 +\int_{\Delta_{Nr}}  n_j (\partial_j u^i -\delta_{ij} \phi ) ( \partial_i \varphi ) (\eta -\beta)\\
 &=J_{21} +J_{22}+ J_{23},
\endaligned
$$
where $n=(n_1, \dots, n_d)$ denotes the outward unit normal to $\partial  E_{Nr}$ and  
the repeated indices $i, j$ are summed from $1$ to $d$.
By H\"older's inequality, 
$$
\aligned
|J_{21} +J_{22} |
&\le  C \| |\nabla u | +|\phi| \|_{L^q(E_{Nr})} \| \eta-\beta \|_{L^{q^\prime}(E_{Nr})} (Nr)^{-2}\\
& \qquad \qquad
+ C  \| |\nabla u | +|\phi | \|_{L^q(E_{Nr})} \| \nabla \eta  \|_{L^{q^\prime}(E_{Nr})} (Nr)^{-1}\\
& \le C  \| |\nabla u | +|\phi | \|_{L^q(E_{Nr})} (Nr)^{-1},
\endaligned
$$
where  $\beta= \fint_{E_{Nr} } \eta$ and we have applied  a  Poincar\'e   inequality.

Finally, we use the trace  inequality, 
\begin{equation}\label{P-ineq}
\left(\int_{\partial E_R } |\eta -\beta|^{p^\prime} \right)^{1/p^\prime}
\le C  \left(\int_{E_{R}} |\nabla \eta|^{q^\prime} \right)^{1/q^\prime},
\end{equation}
with $R=Nr$ and  $\beta =\fint_{E_{R}} \eta $, 
to bound $J_{23}$, where $\frac{d-1}{p^\prime}= \frac{d}{q^\prime}-1$ and $1< q^\prime< d$.
See Remark \ref{re-L1} below  for a proof of \eqref{P-ineq}.  This leads to 
$$
\aligned
|J_{23} |
& \le C \| |\nabla u| + |\phi| \|_{L^p(\Delta_{Nr})} \|\eta-\beta \|_{L^{p^\prime}\Delta_{Nr})} (Nr)^{-1}\\
& \le C \| |\nabla u | + |\phi| \|_{L^p(\Delta_{Nr})}  \|\nabla \eta \|_{L^{q^\prime} (E_{Nr})} (Nr)^{-1}.
\endaligned
$$
As a result, we have proved that
$$
\aligned
|J_1 + J_2|
& \le C \| F\|_{L^\infty(E_{Nr})}  (Nr)^{\frac{d}{q}}
+ C \| |\nabla u| + |\phi  | \|_{L^q(E_{Nr})} (Nr)^{-1}\\
&\qquad\qquad
+ C \| |\nabla u| + |\phi  | \|_{L^p(\Delta_{Nr}) } (Nr)^{-1 }, 
\endaligned
$$
where $p=\frac{q (d-1)}{d}$.
By duality, we see that 
$$
\aligned
|\lambda | \| g \|_{\hW^{-1, q}(E_{Nr} )}
& \le C \| F\|_{L^\infty(E_{Nr})}  (Nr)^{\frac{d}{q}}
+ C \| |\nabla u| + |\phi  | \|_{L^q(E_{Nr})} (Nr)^{-1}\\
&\qquad\qquad
+ C \| |\nabla u| + |\phi  | \|_{L^p(\Delta_{Nr}) } (Nr)^{-1 }.
\endaligned
$$
This, together with \eqref{L-15} as well as the estimates for $\| g \|_{L^q(E_{Nr})}$  and $\| h \|_{L^q(E_{Nr} )}$,
completes  the proof of  \eqref{L-12}.
\end{proof}

\begin{remark}\label{re-L1}
The  trace  inequality \eqref{P-ineq},  with $\beta =\fint_{D_{R}} \eta$ and $1<q^\prime< d$, is fairly 
standard.  We give a proof for the reader's convenience. 
By dilation, it suffices to consider the case $R=1$.
Choose a function $\alpha \in C^1(\R^d; \R^d)$ such that $\alpha \cdot n \ge c>0$ on $\partial E_1$,
where $c$ depends only on $d$ and $M$.
Using integration by parts and H\"older's inequality, we have 
$$
\aligned
c
 \int_{\partial E_1}
|\eta-\beta|^{p^\prime}
 & \le   \int_{\partial E_1}  \alpha \cdot n | \eta -\beta|^{p^\prime}\\
 & \le C \int_{E_1} |\eta-\beta|^{p^\prime}
 + C \int_{E_1} |\eta-\beta|^{p^\prime-1} |\nabla \eta|\\
 &\le C \int_{E_1} |\eta -\beta|^{p^\prime}
 + C \left( \int_{E_1} |\eta-\beta|^{(p^\prime-1) q} \right)^{1/q} \left(  \int_{D_1} |\nabla \eta|^{q^\prime} \right)^{1/q^\prime}\\
 &\le C \left( \int_{E_1} |\nabla \eta|^{q^\prime}\right)^{p^\prime/q^\prime}. 
\endaligned
$$
Note that since $q=\frac{pd}{d-1}$, we have $\frac{1}{(p^\prime-1) q} = \frac{p-1}{q}=\frac{1}{q^\prime}-\frac{1}{d}$. 
The  Sobolev  inequality 
$$
\| \eta -\beta \|_{L^{s} (E_1)} \le C \| \nabla \eta \|_{L^{q^\prime}(E_1)},
$$
 where $1<q^\prime< d$ and 
$\frac{1}{s}=\frac{1}{q^\prime}-\frac{1}{d}$, is used in the last step.
\end{remark}


\section{$L^\infty$ bounds for the case $d=3$}\label{section-V}

Throughout this section we assume $d=3$.

\begin{lemma}\label{lemma-V3}
Suppose  $q>3$ and $r= |\lambda|^{-1/2}$, where $ \lambda\in \Sigma_\theta$.
Let  $(u, \phi)\in W^{1, q} (B_{Nr}; \C^3) \times L^q(B_{Nr}; \C)$ be a weak solution of
\begin{equation}\label{L-11a}
\left\{
\aligned
-\Delta u +\nabla \phi +\lambda u & = F &\quad & \text{ in } B_{Nr},\\
\text{\rm div} (u) & =0 & \quad & \text{ in } B_{Nr},\\
\endaligned
\right.
\end{equation}
where $N\ge 4$ and $F\in L^\infty(E_{Nr}; \C^3)$.
Then
\begin{equation}\label{V3-00}
\aligned
|\lambda| \| u \|_{L^\infty(B_r)}
& \le C \left\{
\| F \|_{L^\infty (B_{Nr})}   N^{\frac{3}{q}}
+ |\lambda| \| u \|_{L^\infty(B_{Nr})} N^{\frac{3}{q}-1}\right\}\\
& +C |\lambda|^{\frac12} \left(\fint_{B_{Nr}} |\nabla u|^q \right)^{1/q}  N^{\frac{3}{q}-1} 
+ C |\lambda|^{\frac12} \inf_{\beta\in \C} \left(\fint_{B_{Nr}} |\phi -\beta |^q \right)^{1/q}  N^{\frac{3}{q}-1},
\endaligned
\end{equation}
where $C$ depends only on $\theta$ and  $q$.
\end{lemma}

\begin{proof}
This is an interior estimate, proved in \cite[Lemma 3.3]{GS2026a} for $q>d$ in any dimension.
\end{proof}

The next lemma is the boundary version of Lemma \ref{lemma-V3}.
Since $d=3$, there exists $q$, depending only on $M$,
 such that $q>d$ satisfies the condition \eqref{q}.

\begin{lemma}\label{lemma-V4}
Let $(u, \phi )$ be the same as in Theorem \ref{thm-L}.
Suppose  $q>3$ satisfies the condition \eqref{q} with $d=3$.
 Let  $r=R |\lambda|^{-1/2}$ for some $R\ge 1$.
Then
\begin{equation}\label{V3-0}
\aligned
|\lambda| \| u \|_{L^\infty(E_r)}
& \le C \left\{
\| F \|_{L^\infty(E_{Nr})}   R N^{\frac{3}{q}}
+ |\lambda| \| u \|_{L^\infty(E_{Nr})} N^{\frac{3}{q}-1}\right\}\\
& 
+ C |\lambda|^{\frac12} 
  N^{\frac{3}{q}-1} \left(\fint_{E_{Nr}}  |\nabla u|^q \right)^{1/q} 
+ C |\lambda|^{\frac12} N^{\frac{3}{q}-1}  \left(\fint_{\Delta_{Nr}} |\nabla u|^p \right)^{1/p} \\
& 
+ C |\lambda|^{\frac12} N^{\frac{3}{q}-1}   \inf_{\beta\in \C}
\left\{ 
 \left(\fint_{E_{Nr}} |\phi -\beta |^q \right)^{1/q}  
 +   \left(\fint_{\Delta_{Nr}} |\phi -\beta |^p \right)^{1/p}  \right\},
\endaligned
\end{equation}
where $p=\frac{2q}{3}$ and  $C$ depends only on $\theta$,   $q$ and $M$.
\end{lemma}

\begin{proof}

Since $q>3=d$, it follows from the Sobolev inequality that 
\begin{equation}\label{V2-1}
\| u \|_{L^\infty(E_r)}
\le C r^{-\frac{3}{q}}
\left\{ \| u \|_{L^q(E_r)}
+  r \| \nabla u \|_{L^q(E_r )} \right\},
\end{equation}
where $C$ depends only on $M$.
Let $r=R|\lambda|^{-1/2}$, where $R\ge 1$. Using  \eqref{V2-1} and \eqref{L-12}, we obtain 
$$
\aligned
|\lambda |  \| u \|_{L^\infty(E_r)}
& \le C r^{-\frac{3}{q}}
\left\{ |\lambda|  \| u \|_{L^q(E_r)}
+  |\lambda|^{1/2} \| \nabla u \|_{L^q(E_r )} \right\}\\
&\le C \Bigg\{
\| F \|_{L^\infty(E_{Nr})}  R N^{\frac{3}{q}}
+|\lambda| \| u \|_{L^\infty(E_{Nr})} N^{\frac{3}{q}-1}\\
 & + \left(\fint_{E_{Nr}} ( |\nabla u| + |\phi|)^q \right)^{1/q} |\lambda|^{\frac12} N^{\frac{3}{q}-1}
+\left(\fint_{\Delta_{Nr}} ( |\nabla u| + |\phi|)^p \right)^{1/p}|\lambda|^{\frac12}
N^{\frac{2}{p}-1} 
\Bigg\}.
\endaligned
$$
Since $\frac{2}{p}=\frac{3}{q}$, this  yields \eqref{V3-0} by replacing $\phi$ with $\phi-\beta$.
\end{proof}

Let $\Omega$ be a bounded  Lipschitz domain.
Then there  exist $r_0>0$ and $M>0$ with the property that  for any $x_0\in \partial\Omega$,
there exists a Cartesian coordinate system with origin at  $x_0$, obtained by translation and rotation, 
such that in the new system, 
\begin{equation}\label{coord}
\aligned
B(0, 16r_0)\cap \Omega
 & = B(0, 16r_0)
\cap \left\{ (x^\prime, x_d)\in \R^d:
x_d > \psi (x^\prime)\right \},\\
B(0, 16r_0)\cap \partial \Omega
 & = B(0, 16r_0)
\cap \left\{ (x^\prime, x_d)\in \R^d:
x_d = \psi (x^\prime)\right \},\\
\endaligned
\end{equation}
where $\psi: \R^{d-1} \to \R$ is a Lipschitz  function with $\|\nabla \psi\|_\infty\le M$ and $\psi (0)=0$.

\begin{lemma}\label{lemma-C-2}
Let $\Omega$ be a bounded Lipschitz  domain in $\R^3$.
Let $(u, \phi )$ be a weak solution of \eqref{eq-0}, where $F\in L^\infty(\Omega; \C^d)$.
Assume that $|\lambda|> r_0^{-2}$. 
Then, if $B(x, 2r) \subset \Omega$ and $r\ge |\lambda|^{-1/2}$, 
\begin{equation}\label{C-2-0}
\left(\fint_{B(x, r)} |\nabla u|^q \right)^{1/q}
 \le C \left\{ |\lambda|^{1/2} \| u \|_{L^\infty(\Omega)} + |\lambda|^{-1/2} \| F \|_{L^\infty(\Omega)}\right\}
 \end{equation}
 for any $q>2$.
 Moreover,  there exists $\e>0$, depending on $\Omega$,
  such that if $x_0\in \partial\Omega$ and $ |\lambda|^{-1/2} \le r < r_0$,  then 
 \begin{equation}\label{C-2-1}
 \aligned
 &  \left( \fint_{B(x_0, r)\cap \Omega} |\nabla u|^q \right)^{1/q}
+
\left( \fint_{B(x_0, r)\cap \partial\Omega} |\nabla u|^p \right)^{1/p}\\
& \qquad
\le C\left\{  |\lambda|^{1/2} \| u \|_{L^\infty(\Omega)} + |\lambda|^{-1/2}  \| F \|_{L^\infty(\Omega)} \right\}, 
\endaligned
\end{equation}
 for  $2< q< 3+\e$ and  $2< p< 2+\e$, where  $C$ depends only on $ p, q$ and $\Omega$.
\end{lemma}

\begin{proof}
The interior estimate \eqref{C-2-0}, which holds in any dimension, is proved in \cite[Lemma 3.6]{GS2026a}.
The boundary estimate in  \eqref{C-2-1}  for 
$$
\left( \fint_{B(x_0, r)\cap \partial\Omega} |\nabla u|^p \right)^{1/p}
$$
with  $2< p< 2+\e$ also holds in  any dimension.
This was noted in \cite[Remark 3.7]{GS2026a}.
 The  proof is the same as in the case of $C^1$ domains, given 
in \cite[Lemma 3.6]{GS2026a}.
To prove the estimate in \eqref{C-2-1} for
$$
 \left( \fint_{B(x_0, r)\cap \Omega} |\nabla u|^q \right)^{1/q}
$$  
with $2< q< 3+\e$, 
we only need to consider the case 
 $r= c |\lambda|^{-1/2}$. 
The general case follows  by covering $B(x_0, r)\cap \Omega$ with a finite number of balls $\{ B(x_\ell, c |\lambda|^{-1/2})\} $, where $x_\ell \in B(x_0, r)\cap \Omega$, 
with a finite overlap.

Finally, let $x_0\in \partial\Omega$ and $r=c |\lambda|^{-1/2}$.
If $d=3$, the range $2< q< 3+\e$ is contained in the interval  given by the condition \eqref{q}. 
As a result, 
it follows from \eqref{local-w2}  that 
$$
\aligned
\left(\fint_{B(x_0, r)\cap \Omega} |\nabla u|^q\right)^{1/q}
 & \le \frac{C}{r}
\left(\fint_{B(x_0, 2r)\cap \Omega} |u|^2\right)^{1/2}
+ C r \| F-\lambda u \|_{L^\infty(\Omega)}\\
 & \le C\left\{  |\lambda|^{1/2} \| u \|_{L^\infty(\Omega)} + |\lambda|^{-1/2}  \| F \|_{L^\infty(\Omega)} \right\},
\endaligned
$$
which completes the proof.
\end{proof}

The following theorem gives an $L^\infty$ bound for $\lambda u$.

\begin{thm}\label{thm-V}
Let  $\Omega$ be a bounded Lipschitz  domain in $\R^3$.
Let $\lambda\in \Sigma_\theta$ and $3<q<3+\e$, where $\e=\e(\Omega)>0$ 
is sufficiently small. 
Let $(u, \phi ) \in W_0^{1, q} (\Omega; \C^d) \times L^q(\Omega; \C) $ be a weak solution of
\begin{equation}\label{V5-0}
\left\{
\aligned
-\Delta u +\nabla p +\lambda u & = F &\quad & \text{ in } \Omega,\\
\text{\rm div} (u) & =0 & \quad & \text{ in } \Omega,\\
u& = 0& \quad & \text{ on } \partial\Omega,
\endaligned
\right.
\end{equation}
where  $F\in L^\infty(\Omega; \C^3)$. Let $t=|\lambda|^{-1/2}$ and $N\ge 4$. Then
\begin{equation}\label{V5-00}
\aligned
  |\lambda| \| u \|_{L^\infty(\Omega)}
 & \le C  N^{20} \| F \|_{L^\infty(\Omega)}  
+ C N^{\frac{3}{q}-1}
 |\lambda| \| u \|_{L^\infty(\Omega)}\\
& + C N^{\frac{3}{q}-1}
|\lambda|^{\frac12}
\sup_{\substack{ t< r< r_0\\ B(x, 2r ) \subset \Omega}}
\inf_{\substack {\beta \in \C }}
\left(\fint_{B(x, r ) } | \phi  -\beta|^q \right)^{1/q}\\
&+ C N^{\frac{3}{q}-1}
|\lambda|^{\frac12}
\sup_{\substack{ x_0\in \partial \Omega \\ t< r< r_0 } }
\inf_{\substack {\beta \in \C }}
\left\{
\left(\fint_{B(x_0, r )\cap \Omega } | \phi  -\beta|^q \right)^{1/q}
+ \left(\fint_{B(x_0, r  )\cap \partial\Omega}
|\phi  -\beta|^p \right)^{1/p} \right\},
\endaligned
\end{equation}
where $\frac{2}{p}=\frac{3}{q}$ and $C$ depends only on $\theta$ and $\Omega$.
\end{thm}

\begin{proof}

First, since the boundary $W^{1, q}$ estimates  for the stationary Stokes equations in Lipschitz domains hold for some
$q>3$ in dimension $3$, the argument in the proof of Theorem 3.9 in \cite{GS2026a} gives
\begin{equation}
\| u \|_{L^\infty(\Omega)}\label{u-i}
\le C (1+|\lambda|)^4 \| F \|_{L^\infty(\Omega)}.
\end{equation}
As a result, we only need to consider the case where $|\lambda|> C^2 r_0^{-2} N^4$ and $C=C(M)$ is  large.

Next, let $r=2N|\lambda|^{-1/2}$.
It follows from Lemma \ref{lemma-V4} by translation and rotation that for any $x_0 \in \partial\Omega$,
\begin{equation}\label{V5-5}
\aligned
& |\lambda| \| u \|_{L^\infty(B(x_0, r) \cap \Omega)}\\
& \le C \left\{
\| F \|_{L^\infty(\Omega)}   N^{\frac{3}{q} +1 } 
 +   |\lambda| \| u \|_{L^\infty(\Omega)} N^{\frac{3}{q}-1}\right\}\\
&\qquad
+C |\lambda|^{\frac12}  N^{\frac{3}{q}-1}  \left\{
 \left(\fint_{B(x_0, CNr)\cap \Omega} |\nabla u|^q \right)^{1/q}
 + \left(\fint_{B(x_0, CNr)\cap \partial\Omega}  |\nabla u|^p \right)^{1/p} \right\}\\
 & + C |\lambda|^{\frac12} N^{\frac{3}{q}-1}  \inf_{\beta\in \C}
\left\{ 
 \left(\fint_{B(x_0, CNr)\cap \Omega} |\phi -\beta |^q \right)^{1/q}  
 + \left(\fint_{B(x_0,C Nr) \cap \partial\Omega} |\phi -\beta |^p \right)^{1/p}  \right\}.
\endaligned
\end{equation}
In view of  \eqref{C-2-1}, we obtain 
\begin{equation}\label{V5-6}
\aligned
& |\lambda| \| u \|_{L^\infty (\Omega_r )}\\
& \le C \left\{
\| F \|_{L^\infty(\Omega)}   N^{\frac{3}{q} +1 }
+|\lambda| \| u \|_{L^\infty(\Omega)} N^{\frac{3}{q}-1}\right\}\\
& 
+ C |\lambda|^{\frac12} N^{\frac{3}{q}-1}  \sup_{x_0\in \partial\Omega} \inf_{\beta\in \C}
\left\{ 
 \left(\fint_{B(x_0, CNr)\cap \Omega} |\phi -\beta |^q \right)^{1/q}  
 + \left(\fint_{B(x_0, CNr) \cap \partial\Omega} |\phi -\beta |^p \right)^{1/p}  \right\},
\endaligned
\end{equation}
where $\Omega_r = \{ x\in \Omega: \text{\rm dist}(x, \partial\Omega)< r \}$ and $r=2N|\lambda|^{-1/2}$.

Finally, for $x \in \Omega\setminus \Omega_r$, we have  $B(x, r) = B(x, 2N |\lambda|^{-1/2} )\subset \Omega$.
By applying  Lemma \ref{lemma-V3}  and using \eqref{C-2-0}, we see that 
\begin{equation}\label{V5-7}
\aligned
|\lambda| \| u \|_{L^\infty(B(x, |\lambda|^{-1/2} ))}
 & \le C \left\{ 
\| F \|_{L^\infty(\Omega)}  N^{\frac{3}{q}}
+ |\lambda| \| u \|_{L^\infty(\Omega)} N^{\frac{3}{q}-1}\right\}\\
&\qquad
+CN^{\frac{3}{q}-1} |\lambda|^{\frac12}
\inf_{\beta\in \C}
\left(\fint_{B(x, N|\lambda|^{-1/2})}
|\phi  -\beta|^q \right)^{1/q}.
\endaligned
\end{equation}
We obtain \eqref{V5-00} by combining \eqref{V5-6} and   \eqref{V5-7}. 
\end{proof}



\section{Improved estimates for the pressure}\label{section-N}

Throughout this section, unless otherwise indicated,
 we assume that  $\Omega$ is a bounded Lipschitz domain in $\R^d$, $d\ge 2$.
Consider  the Neumann problem for Laplace's equation,
\begin{equation}\label{NP0}
\left\{
\aligned
 \Delta \phi & =0 & \quad & \text{ in } \Omega,\\
 \frac{\partial \phi}{\partial n} &= (n_i \partial_j -n_j \partial_i) g_{ij} & \quad & \text{ on } \partial\Omega,
\endaligned
\right.
\end{equation}
where the repeated indices $i, j$ are summed from $1$ to $d$ and $g_{ij} \in C^1(\partial\Omega)$. 

Let $g=(g_{ij})$ and  $I(x_0, r)= B(x_0, r) \cap \partial\Omega$.
For $0< r<   r_0$,  define
\begin{equation}\label{M-t}
M_q (g)  (r) =\sup_{\substack{x_0 \in \partial\Omega \\ r\le  s< r_0}}
\left(\fint_{I(x_0, s)} |g|^q \right)^{1/q}.
\end{equation}

\begin{lemma}\label{lemma-P1}
Let $\phi  \in C^1(\overline{\Omega})$ be a weak solution of  \eqref{NP0}.
Let $2\le p< \infty$.
Then, for $0< r< r_0$, 
\begin{equation}\label{P1-0}
\sup_{\substack{x_0 \in \partial\Omega }}
\left(\fint_{I(x_0, r)}
|\phi  -\fint_{I(x_0, r)} \phi |^p \right)^{1/p}
\le C M_p (g) (r),  
\end{equation}
where $C$ depends only on $d$, $p$ and the Lipschitz character of $\Omega$.
\end{lemma}

\begin{proof} See \cite[Lemma 4.1]{GS2026a}.
\end{proof}

The next lemma improves the estimates in \cite[Lemma 4.3]{GS2026a}.

\begin{lemma}\label{lemma-N1}
Let  $2\le p< \infty$ and $q=\frac{pd}{d-1}$.
Let $\phi$ be the same as in Lemma \ref{lemma-P1}.
Then, for $0< r< r_0$,
\begin{equation}\label{N1-0}
\sup_{x_0 \in \partial\Omega} 
\left(\fint_{B(x_0, r) \cap \Omega} |\phi -\fint_{I(x_0, r)} \phi |^q\right)^{1/q}
\le C M_p(g) (r).
\end{equation}
Moreover, for any $x\in \Omega$,
\begin{equation}\label{N1-1}
 \delta(x) |\nabla \phi (x) | 
\le C M_2 (g) (\delta (x) ), 
\end{equation}
 where $\delta (x)=  \text{\rm dist}(x, \partial\Omega)$.
 The constants  $C$ in \eqref{N1-0}-\eqref{N1-1}  depend at most on $d$, $p$ and the Lipschitz character of $\Omega$.
 \end{lemma}

\begin{proof}

The estimate \eqref{N1-1} was proved in \cite[Lemma 4.3]{GS2026a}.
The inequality \eqref{N1-0}  with $p$ in the place of $q$ was also proved  in \cite[Lemma 4.3]{GS2026a}.
To establish \eqref{N1-0} with $q=\frac{pd}{d-1}$, 
we let $\beta=\fint_{I(x_0, r)} \phi $, where $x_0\in \partial\Omega$ and $0< r< r_0$.
Write $\phi-\beta =\phi_1 +\phi_2$, where $\phi_1$ and $\phi_2$ are bounded harmonic functions in $\Omega$ and
$\phi_1 = (\phi - \beta)\chi_{I(x_0, 4r) }$ on $\partial\Omega$.
Then
\begin{equation}\label{p-1}
\aligned
\left(\fint_{B(x_0, r)\cap \Omega} 
|\phi_1|^q \right)^{1/q}
 & \le C r^{\frac{1-d}{p}} \|  (\phi_1)^* \|_{L^p(I(x_0, 2r) )}
\le C  r^{\frac{1-d}{p}} \| \phi_1 \|_{L^p(\partial\Omega)}\\
&\le C r^{\frac{1-d}{p}} \| \phi-\beta \|_{L^p(I(x_0, 4r))}
\le C  M_p(g) (r),
\endaligned
\end{equation}
where $(\phi_1)^*$ denotes the non-tangential  maximal function of $\phi_1$ and we have used Lemma \ref{lemma-P1} for 
the last step.

The second inequality in \eqref{p-1} follows from a classical work of
B. Dahlberg \cite{Dahlberg-1979} for harmonic functions in Lipschitz domains. 
 To see the first inequality in \eqref{p-1}, we assume $r=1$ by rescaling.
 Observe  that  for $x\in B(x_0, 1) \cap\Omega$,
\begin{equation}\label{ii-1}
|\phi_1(x)|\le C  \int_{I(x_0, 2)} \frac{(\phi_1)^* (y)}{|x-y|^{d-1}}\, dy.
\end{equation}
Let $G\in L^{q^\prime}(B(x_0, 1)\cap \Omega)$ and 
$$
H(x)= \int_{B(x_0, 1) \cap \Omega}
\frac{ |G(x)| }{|x-y|^{d-1}} \, dx.
$$ 
 By Fubini's Theorem and \eqref{ii-1}, 
$$
\aligned
\Big| \int_{B(x_0, 1)\cap \Omega} \phi_1 G  \Big|
&\le C \int_{I(x_0, 2)}  H(x)  (\phi_1)^* (x) d x\\
& \le C \| (\phi_1)^* \|_{L^p(I(x_0, 2))} \| H \|_{L^{p^\prime} (I(x_0, 2))}.
\endaligned
$$
As in Remark \ref{re-L1}, we have
$$
\aligned
\int_{I(x_0, 2)} |H|^{p^\prime}
 & \le C \int_{B(x_0, 2) \cap \Omega} |H|^{p^\prime}
+ C \int_{B(x_0, 2) \cap \Omega} |H|^{p^\prime-1} |\nabla H|\\
& \le C \int_{B(x_0, 2) \cap \Omega} |H|^{p^\prime}
+ C \left(\int_{B(x_0, 2)\cap \Omega}
|H|^{(p^\prime-1)q} \right)^{1/q}
\left(\int_{B(x_0, 2)\cap \Omega} |\nabla H|^{q^\prime} \right)^{1/q^\prime}\\
&\le C \left(\int_{B(x_0, 1)\cap \Omega)} |G|^{q^\prime}\right)^{p^\prime/q^\prime},
\endaligned
$$
where we have used the fractional integral estimate
$\|H \|_{L^{(p^\prime-1)q} (\R^d)} \le C \| G \|_{L^{q^\prime}(B(x_0,1 ) \cap \Omega)}$ and
the Calder\'on-Zygmund estimate
$\|\nabla H \|_{L^{q^\prime} (\R^d)}
\le C  \| G \|_{L^{q^\prime}(B(x_0, 1) \cap \Omega)}$.
As a result, we have proved that 
$$
\Big| \int_{B(x_0, 1)\cap \Omega} \phi_1 G  \Big|
 \le C \| (\phi_1)^* \|_{L^p(I(x_0, 2))} \| G \|_{L^{q^\prime} ( B(x_0, r_0))}.
$$
By duality we obtain 
$$
\| \phi_ 1\|_{L^q (B(x_0, 1)\cap \Omega)}
\le C  \| (\phi_1)^* \|_{L^p(I(x_0, 2))},
$$
which is the first inequality in \eqref{p-1} with $r=1$.

Finally, using representation by a Poisson integral,
  it was proved in \cite[Lemma 4.3]{GS2026a} that for any $x\in B(x_0, r) \cap \Omega$, 
\begin{equation}\label{P2-2}
|\phi_2 (x)| \le C  M_p(g) (r).
\end{equation}
The desired estimate \eqref{N1-0} follows from \eqref{p-1} and \eqref{P2-2}.
\end{proof} 

\begin{thm}\label{thm-N1}
Let $\Omega$ be a bounded Lipschitz domain in $\R^d$, $d\ge 2$.
Let $\phi \in C^1(\overline{\Omega})$  be a solution of \eqref{NP0}.
Then, 
\begin{equation}\label{P3-0}
\sup_{B(x, 2r)\subset \Omega}
\left(\fint_{B(x, r)} |\phi - \fint_{B(x, r)} \phi |^q \right)^{1/q}
\le C M_2 (g) (r)
\end{equation}
for any $2< q< \infty$. Moreover, 
for $x_0\in \partial\Omega$ and $0<r< r_0$,
\begin{equation}\label{P3-00}
\left(\fint_{B(x_0, r) \cap \partial\Omega}
|\phi -\fint_{B(x_0, r) \cap \partial\Omega} \phi |^p \right)^{1/p}
+
\left(\fint_{B(x_0, r) \cap \Omega}
|\phi -\fint_{B(x_0, r) \cap\partial \Omega} \phi  |^q \right)^{1/q}
\le C  M_p (g) (r), 
\end{equation}
where $2\le p< \infty$ and  $q=\frac{pd}{d-1}$. The constant $C$ 
depends only on $d$, $p$  and the Lipschitz character of $\Omega$.
\end{thm}

\begin{proof}
The estimate \eqref{P3-0} follows from \eqref{N1-1} by using a Poincar\'e inequality,
while \eqref{P3-00} is contained in Lemmas \ref{lemma-P1} and \ref{lemma-N1}.
\end{proof}

We now use Theorem \ref{thm-N1} to bound the pressure terms in \eqref{V5-00} for $d=3$.

\begin{thm}\label{thm-P1}
Let $d=3$ and $\lambda\in \Sigma_\theta$.
Let $(u, \phi )$ be a weak solution of \eqref{eq-0}, where
$F\in C_{0, \sigma}^\infty(\Omega)$ and $\Omega$ is a bounded Lipschitz domain.
Let  $B(x, 2r)\subset \Omega$.
Then
\begin{equation}\label{P-99}
\left(
\fint_{B(x, r)} |\phi -\fint_{B(x, r)}  \phi  |^q \right)^{1/q}
\le C  M_2(|\nabla u|)(\delta(x) )
\end{equation}
for any $2< q< \infty$, 
where $C$ depends only on  $\Omega$.
\end{thm}

\begin{thm}\label{thm-P2}
Let $d=3$ and $\lambda\in \Sigma_\theta$.
Let $(u, \phi )$ be the same as in Theorem \ref{thm-P1}.
Suppose $x_0 \in \partial\Omega$ and $0< r<  r_0$.
Then, for $2\le q< 3+\e$ and $p=\frac{2q}{3}$, 
\begin{equation}\label{P-97}
\left(
\fint_{B(x_0, r)\cap \Omega} |\phi -\beta|^q \right)^{1/q}
+\left(\fint_{B(x_0, r)\cap \partial\Omega}
|\phi -\beta|^p \right)^{1/p} 
\le C M_p(|\nabla u|) (r),
\end{equation}
where 
 $\beta=\fint_{B(x_0, r) \cap \partial\Omega} \phi $ and  $C$ depends only on $q$ and $\Omega$.
\end{thm}

\begin{proof}

With the estimates in Theorem \ref{thm-N1} at our disposal,
the proofs of Theorems \ref{thm-P1} and \ref{thm-P2}, using an approximation argument, 
 are similar to  that of 
Theorems 6.1 an 6.2 in \cite{GS2026a}. 
We should point out that  Theorems 6.1 and 6.2 in \cite{GS2026a} contain
 the additional assumptions that $u\in C^1(\overline{\Omega}, \C^d)$ and $\phi  \in C(\overline{\Omega}; \C)$.
 Without the smoothness assumptions,
 we use the fact that 
 $u \in W^{1, q}(\Omega; \C^3)$ for some $q>3$ if $d=3$, which follows from
    the regularity estimates for the stationary Stokes equations \cite{BS-1995}.
 As a result,   $u$ is continuous up to the boundary and
  $\| u\|_{L^\infty(\partial \Omega_\ell) } \to 0$, where $\{ \Omega_\ell\}$ is a sequence of 
Lipschitz domains  with uniform Lipschitz characters, which approximate $\Omega$ from inside.
Also, by using  (A-8) in \cite{GS2026a},  we deduce that 
$$
M_p^\ell (|\nabla u|) (r ) \to  \ M_p(|\nabla u|)(r)
$$
 for $2< p< 2+\e$, where $M_p^\ell$ is the analogy of $M_p$ on $\partial\Omega_\ell$.
 The details are omitted.
\end{proof}


\section{Proof of Theorem \ref{main-1}} \label{section-M1}

The proof for the endpoint case $p=\infty$ is similar to that of Theorem 1.1 in \cite{GS2026a} for $C^1$ domains.
Let $\Omega$ be a bounded Lipschitz  domain in $\R^3$ and $\lambda\in \Sigma_\theta$.
Let $(u, \phi)$ be a weak solution of \eqref{eq-0} with $F\in L^\infty_\sigma (\Omega)$.
We need to show that 
\begin{equation}\label{f-i}
|\lambda| \| u \|_{L^\infty(\Omega)} \le C \| F \|_{L^\infty(\Omega)}.
\end{equation}
The  case $|\lambda| \le C r_0^{-2}$ is given by \eqref{u-i}, which follows from the regularity estimates 
 for the stationary Stokes equations in a Lipschitz domain in $\R^3$.
 
 To establish the estimate \eqref{f-i} for large  $|\lambda|$. 
We first consider the case $F\in C_{0, \sigma}^\infty(\Omega)$.
It follows  by Lemma \ref{lemma-C-2} that 
\begin{equation}\label{main-B1}
M_p (|\nabla u|) (r)\le C \left\{ |\lambda|^{1/2} \| u \|_{L^\infty(\Omega)} + |\lambda|^{-1/2}  \| F \|_{L^\infty(\Omega)} \right\}
\end{equation}
for $2<  p< 2+\e $ and  $|\lambda|^{-1/2}\le r< r_0$, where $\e>0$ depends on $\Omega$. 
This, together with 
   \eqref{V5-00}, \eqref{P-99} and \eqref{P-97} with $q=\frac{3p}{2}$, implies   that for any $N\ge 2$,
\begin{equation}\label{M-1}
\aligned
  |\lambda| \| u \|_{L^\infty (\Omega)}
   \le C_1 N^{{100}} \| F \|_{L^\infty (\Omega)}
 + C_1 N^{-1/2} |\lambda | \| u \|_{L^\infty(\Omega)}, 
 \endaligned
 \end{equation}
 where $C_1$ depends only on  $\Omega$.
 By choosing $N$ so large that $C_1 N^{-1/2} \le (1/2)$, we obtain \eqref{f-i}.
  The general case $F\in L^\infty_\sigma(\Omega)$ follow  by an approximation argument, as in \cite{AG-2013, AG-2014, GS2026a}. 
By applying Theorem \ref{thm-I} in the Appendix, we obtain the resolvent estimate \eqref{est-0} for $2< p< \infty$.

For $p< 2$, we use a duality argument.
Let $u\in W^{1, 2}_0 (\Omega; \C^3)$ be a weak solution of \eqref{eq-0} with $F\in C_{0, \sigma}^\infty (\Omega)$.
Let $v\in W^{1, 2}_0(\Omega; \C^3)$ be a weak solution of \eqref{eq-0} with $G\in L^{p^\prime}_\sigma (\Omega)$
in the place of $F$. It follows by the definition of weak solutions that 
\begin{equation}\label{M1-1}
\aligned
\Big| \int_\Omega u \cdot G  \Big| & =\Big| \int_\Omega v \cdot F\Big|
\le \| v\|_{L^{p^\prime}(\Omega)}  \| F \|_{L^p(\Omega)}\\
& \le C |\lambda |^{-1}  \| G \|_{L^{p^\prime}(\Omega)} \| F \|_{L^p(\Omega)},
\endaligned
\end{equation}
where  we have used the resolvent estimate in $L_\sigma^{p^\prime}$ for $p^\prime>2$.
Under the assumption that the $L^{p^\prime}$-Helmholtz decomposition holds, for any $H\in L^{p^\prime}(\Omega; \C^3)$,
there exist $w\in L^{p^\prime}_\sigma (\Omega)$ and  $\xi \in W^{p^\prime} (\Omega; \C)$
such that  $H =w +\nabla \xi$ and
\begin{equation}\label{M1-2}
\| w \|_{L^{p^\prime}(\Omega)} + \|\nabla \xi \|_{L^{p^\prime}(\Omega)}
\le C  \| H \|_{L^{p^\prime}(\Omega)}.
\end{equation}
Therefore,
$$
\aligned
\Big| \int_\Omega u \cdot H  \Big| 
& = \Big| \int_\Omega u \cdot  w \Big| \\
& \le C   |\lambda |^{-1}  \| w \|_{L^{p^\prime}(\Omega)} \| F \|_{L^p(\Omega)}\\
&\le C   |\lambda |^{-1}  \| H \|_{L^{p^\prime}(\Omega)} \| F \|_{L^p(\Omega)}
\endaligned
$$
for any $H \in L^{p^\prime} (\Omega; \C^3)$, where we have used \eqref{M1-1} for the first inequality
and \eqref{M1-2} for the last.
By duality  this yields $ \| u \|_{L^p(\Omega)} \le C |\lambda|^{-1} \| F \|_{L^p(\Omega)}$.

Note that  the $L^{p^{\prime}}$-Helmholtz decomposition holds for $(3/2)-\e < p< 2$ \cite{Fabes1998}. As a result, 
we obtain the resolvent estimate \eqref{est-0} for $(3/2)-\e< p< 2$ if $d=3$ and $\Omega$ is Lipschitz.
The same argument also gives the estimate \eqref{est-0} for $1< p< 2$
  if $\Omega$ is a bounded convex domain in $\R^3$, as the $L^p$-Helmholtz decomposition in 
  $\Omega$  holds for any  $1< p< \infty$ \cite{GS2010}.
  
  Finally, by Theorem \ref{thm-c2}, the range $(3/2)-\e< p\le \infty$ is sharp for Lipschitz domains in $\R^3$.


\section{A resolvent estimate in H\"older spaces}

Let $u$ be a weak solution of \eqref{eq-0} with $F\in L^\infty_\sigma (\Omega)$.
Motivated by \cite{BG2025},  consider the resolvent estimate
\begin{equation}\label{est-H1}
|\lambda|^{1-\frac{\alpha}{2}} \| u \|_{C^{0, \alpha}(\Omega)}
\le C \| F \|_{L^\infty(\Omega)}, 
\end{equation}
where $0< \alpha< 1$ and 
$$
\|u \|_{C^{0, \alpha}(\Omega)}=
\sup_{x, y \in \Omega, x\neq y}
\frac{| u(x)- u(y)|}{|x-y|^\alpha}.
$$

\begin{thm}\label{thm-H1}
Let $\Omega$ be a bounded Lipschitz domain in $\R^d, d\ge 2$. 
Let $\lambda\in \Sigma_\theta$, where $\theta\in  (0, \pi/2)$.
Then for any $F\in L^\infty_\sigma (\Omega)$, the weak solution of \eqref{eq-0}
satisfies \eqref{est-H1} for 
\begin{equation}\label{a-range}
\left\{
\aligned
& 0< \alpha< \alpha_0  & \quad &\  \text{ if } d=3,\\
& 0< \alpha< \frac12 +\alpha_0 &  \quad &\  \text{ if } d=2,\\
& 0< \alpha < 1 & \quad &\  \text{ if } d\ge 2 \text{ and } \Omega \text{ is } C^1,
\endaligned
\right.
\end{equation}
where $\alpha_0>0$ depends on $\Omega$.
\end{thm}

\begin{proof}

The proof uses the resolvent estimate in $L^\infty$ as well as the $C^\alpha$ estimate
for the stationary Stokes equations.
Indeed, note that the localized $W^{1, q}$ estimate \eqref{local-w2} for the stationary Stokes equations holds
for 
\begin{equation}
\left\{
\aligned
& 2< q< 3+\e  & \quad &\  \text{ if } d=3,\\
& 2< q< 4+\e &  \quad &\  \text{ if } d=2,\\
& 2< q< \infty  & \quad &\  \text{ if } d\ge 2 \text{ and } \Omega \text{ is } C^1,
\endaligned
\right.
\end{equation}
where $\e>0$ depends on $\Omega$.
By Sobolev imbedding it follows that for $x_0 \in \Omega$ and $0< r< r_0$, 
\begin{equation}\label{H1a}
\| u \|_{C^{0, \alpha}(\Omega \cap B(x_0, r))}\\
 \le \frac{C}{r^\alpha}
\left\{ \| u \|_{L^\infty(\Omega)}
+ r^2 \| F -\lambda u \|_{L^\infty(\Omega)} \right\},
\end{equation}
 where  $\alpha\in (0, 1)$ satisfies the condition \eqref{a-range} and $u$ is a solution of \eqref{eq-0}.
In particular, this implies that  if  $|\lambda|\le C r_0^{-2}$, 
\begin{equation}\label{C-local}
\aligned
\| u \|_{C^{0, \alpha}(\Omega)}
 & \le C  r_0^{-\alpha}  \left\{ \| u \|_{L^\infty(\Omega)} +  r_0^2 \| F \|_{L^\infty(\Omega)} \right\}\\
 & \le C |\lambda|^{\frac{\alpha}{2}-1} \| F \|_{L^\infty(\Omega)}, 
 \endaligned
\end{equation}
where we have used the estimate $\|u \|_{L^\infty(\Omega)} \le Cr_0^2   \| F \|_{L^\infty(\Omega)}$.

To handle the case $|\lambda|> C r_0^{-2}$, 
let  $x, y\in \Omega$ and $r=|x-y|> 0 $.
If $ r\ge |\lambda|^{-1/2}$, it follows by the $L^\infty$ resolvent estimate that 
$$
\aligned
\frac{| u(x)- u(y)|}{|x-y|^\alpha}
 & \le 2 |\lambda|^{\frac{\alpha}{2}} \| u \|_{L^\infty(\Omega)}\\
 &\le C  |\lambda|^{\frac{\alpha}{2}-1} \| F \|_{L^\infty(\Omega)}.
 \endaligned
$$
Finally, suppose $r<|\lambda|^{-1/2}$.
Since $|\lambda|^{-1/2}< cr_0$,  in view of \eqref{H1a}, we obtain 
$$
\aligned
\frac{| u(x)- u(y)|}{|x-y|^\alpha}
& \le \| u \|_{C^{0, \alpha}(\Omega \cap B(x, 2|\lambda|^{-1/2} ))}\\
& \le  C |\lambda|^{\frac{\alpha}{2}}
\left\{ \| u \|_{L^\infty(\Omega)}
+ |\lambda|^{-1}  \| F -\lambda u \|_{L^\infty(\Omega)} \right\}\\
& \le  C |\lambda|^{\frac{\alpha}{2}}
\left\{ \| u \|_{L^\infty(\Omega)}
+ |\lambda|^{-1}  \| F \|_{L^\infty(\Omega)} \right\}\\
 & \le  C  |\lambda|^{\frac{\alpha}{2}-1} \| F \|_{L^\infty(\Omega)}.
\endaligned
$$
This completes the proof.
\end{proof}




 \appendixpage


  \appendix


\section{An interpolation theorem of solenoidal $L^p$ spaces}

Let $\Omega$ be a bounded Lipschitz domain in $\R^d, d\ge 2$.
For $1\le   p\le \infty$, a weak formulation for $L^p_\sigma (\Omega)$ is given by
\begin{equation}\label{L-p}
L^p_\sigma(\Omega)=\Big\{ u \in L^p(\Omega; \C^d): \ 
\int_\Omega u \cdot \nabla \varphi =0 \text{ for any } \varphi \in C_0^\infty(\R^d; \C^d) \Big\}.
\end{equation}

\begin{thm}\label{thm-I}
Let $\Omega$ be a bounded Lipschitz domain in $\R^d, d\ge 2$.
Let $T$ be a bounded linear operator from $L^2_\sigma(\Omega)$ to $L^2(\Omega; \C^d)$.
Suppose that 
\begin{equation}\label{A-00}
\aligned
\| Tu \|_{L^2(\Omega)} &  \le M_2 \| u \|_{L^2(\Omega)}  & \quad & \text{ for } u \in L^2_\sigma(\Omega),\\
\| Tu \|_{L^\infty(\Omega)} & \le M_\infty  \| u \|_{L^\infty(\Omega)}  & \quad & \text{ for   } u \in C^\infty_{0, \sigma}(\Omega).
\endaligned
\end{equation}
Then $T$ is bounded from $L^p_\sigma(\Omega)$ to $L^p(\Omega; \R^d)$ for  any $2< p\le  \infty$, and 
\begin{equation}\label{A-0a}
\| Tu \|_{L^p(\Omega)} \le C_p M_2^\theta M_\infty^{1-\theta}  \| u \|_{L^p(\Omega)}
\end{equation}
for any $u\in L^p_\sigma (\Omega)$, 
where $\theta=\frac{2}{p}$ and $C_p$ depends only on $p$ and $\Omega$.
\end{thm}

\begin{proof}

We divide the proof into four steps.

Step 1. We show that $T$ is bounded from $L^\infty_\sigma (\Omega)$ to $L^\infty(\Omega; \C^d)$.

Let $u \in L^\infty_\sigma(\Omega)$. 
There exists a sequence $\{ u_k \}\subset C_{0, \sigma}^\infty(\Omega)$ such that
$u_k \to u$ a.e. in $\Omega$ and
$\| u_k \|_{L^\infty(\Omega)} \le C \| u\|_{L^\infty (\Omega)}$ \cite{AG-2013}.
As $\Omega$ is bounded, we also have $u_k \to u$ in $L^2_\sigma(\Omega)$.
Since $T$ is bounded from $L^2_\sigma (\Omega)$ to $L^2(\Omega; \C^d)$,
it follows that $Tu_k \to Tu $ in $L^2(\Omega; \C^d)$.
By passing to a subsequence, we may assume that 
$Tu_k \to Tu$ a.e. in $\Omega$.
Note that by the second inequality in \eqref{A-00}, 
$$
\|Tu_k \|_{L^\infty(\Omega)}
\le M_\infty \| u_k \|_{L^\infty(\Omega)}
\le CM_\infty \| u \|_{L^\infty(\Omega)}.
$$
By letting $k \to \infty$, we obtain 
$\| Tu \|_{L^\infty(\Omega)} \le C M_\infty \| u \|_{L^\infty(\Omega)}$.

\medskip

Step 2. For $u=(u_1, \dots, u_d)$ on $\Omega$, let $\widetilde{u}$ be its zero extension to $\R^d$.
Consider  the $(d-1)$-form, 
$$
\omega_u =\sum_{j=1}^d (-1)^{j-1} \widetilde{u}_j dx_1\wedge\cdots \wedge \widehat{dx_j}\wedge\cdots \wedge dx_d,
$$
where the notation $\widehat{dx_j}$ means that the indicated factor is  omitted.
With the Euclidean coefficient norm, $|\omega_u|=|\widetilde{u}|$, we have
$\| \omega_u \|_{L^p(\R^d)}=\| u \|_{L^p(\Omega)}$.
Note that if $u\in L^p_\sigma (\Omega)$,
$$
\langle \text{div}(\widetilde{u}), \varphi\rangle
=-\int_\Omega u \cdot  \nabla \varphi =0
$$
for any $\varphi \in C_0^\infty(\R^d, \C^d)$.
For $1\le p\le \infty$, set
$$
\mathcal{Z}_p
=\left\{ \omega \in L^p (\R^d; \Lambda^{d-1}): \ d\omega =0 \text{ in } \R^d \text{ and }
\text{supp}  (\omega)  \subset \overline{\Omega} 
\right\}.
$$
Since $d\omega_u =  \text{\rm div} (\widetilde{u} )  dx_1 \wedge\cdots \wedge dx_d$,
it follows that the map $u \to \omega_u$ is an isometric isomorphism  from 
$L^p_\sigma(\Omega)$ to $\mathcal{Z}_p$ for $1\le  p\le \infty$.
Define 
$$
\widetilde{T}: \mathcal{Z}_2 \to L^2(\Omega; \C^d)
$$
 by
$\widetilde{T} \omega = T u$ for $\omega=\omega_u$.
Thus the inequality \eqref{A-0a} is equivalent to
\begin{equation}\label{A-2a}
\| \widetilde{T} \omega \|_{L^p(\Omega)} \le C M_2^\theta M_\infty^{1-\theta} \| \omega\|_{L^p(\R^d)}
\end{equation}
for any $\omega \in \mathcal{Z}_p$, where  $2< p< \infty$ and  $\theta =\frac{2}{p}$.

\medskip

Step 3. 
 For $1< p\le  \infty$, define 
$$
\mathcal{V}_p 
=\left\{ A \in W^{1, p} (\R^d; \Lambda^{d-2}): \ \text{supp} (A) \subset \overline{\Omega} \right\}.
$$
If $1< p< \infty$, $\mathcal{V}_p$ can be identified with $W_0^{1, p}(\Omega; \Lambda^{d-2})$.
Since $A=0$ outside a bounded set, it follows by Poincar\'e's  inequality that 
$$
\| A \|_{W^{1, p}(\R^d)} \approx \| DA \|_{L^p(\R^d)},
$$
where $DA$ denotes the full coefficient gradient of $A$.
By the compact-support de Rham theorem of Costabel and McIntosh \cite[Theorem 1.1]{CM}, there
exists a finite-dimensional space  $\mathcal{H}$  such that
$$
\mathcal{H}\subset \left\{ h \in C^\infty(\R^d;  \Lambda^{d-1}): \ \text{supp}(h) \subset \overline{\Omega} \right\},
$$
and  for every $1< p< \infty$, 
\begin{equation}\label{A-0b}
\mathcal{Z}_p =d \mathcal{V}_p \oplus \mathcal{H}.
\end{equation}
This means that every $\omega \in \mathcal{Z}_p$ admits a representation 
\begin{equation}\label{A-1a}
\omega= dA + h, \quad A\in \mathcal{V}_p, \ h \in \mathcal{H},
\end{equation}
with
\begin{equation}\label{A-1b}
\| DA \|_{L^p(\R^d)} + \| h \|_{L^p(\R^d)} \le C \| \omega \|_{L^p(\R^d)}.
\end{equation}
We point out that the finite-dimensional space $\mathcal{H}$ depends on $\Omega$, but not on $p$.
Note that 
$$
\aligned
\| \widetilde{T} (h) \|_{L^p(\Omega)}
 & \le  \|\widetilde{T}(h) \|^\theta _{L^2 (\Omega)} \|\widetilde{T}(h) \|^{1-\theta}_{L^\infty(\Omega)}\\
& \le  (  M_2 \| h \|_{L^2(\R^d)} )^\theta 
( M_\infty \| h \|_{L^\infty(\R^d)})^{1-\theta}\\
& \le C M_2^\theta M_\infty^{1-\theta} \| h \|_{L^p(\R^d)}\\
& \le   C M_2^\theta M_\infty^{1-\theta} \| \omega  \|_{L^p(\R^d)},
\endaligned
$$
where we have used \eqref{A-1b} as well as the fact that the norms in a finite-dimensional space 
are equivalent.
As a result, it suffices to show that  for any $A\in \mathcal{V}_p$,
\begin{equation}\label{A-10}
\| \widetilde{T} (dA) \|_{L^p(\Omega)}
\le C M_2^\theta M_\infty^{1-\theta} \| DA  \|_{L^p(\R^d)}.
\end{equation}

\medskip

Step 4. We use a Lipschitz truncation to complete the proof.

For $A\in \mathcal{V}_p$,  define
$$
g(x)=\mathcal{M} (|DA|) (x), 
$$
where $\mathcal{M}$ is the Hardy-Littlewood maximal operator on $\R^d$.
For $\lambda>0$, let 
$$
E_\lambda =\left\{ x\in \Omega: \ x \text{ is  a Lebesgue point of } A \text{ and } g(x) \le \lambda \right\},
$$
and
$$
G_\lambda =E_\lambda \cup \Omega^c.
$$
Define $a_\lambda$ on $G_\lambda$ by
$$
a_\lambda (x) =
\left\{ 
\aligned
& A(x) & \quad & \text{ for }  x\in E_\lambda,\\
& 0 & \quad & \text{ for }  x\in \Omega^c.
\endaligned
\right.
$$
We will show that $a_\lambda$ is Lipschitz on $G_\lambda$ with constant at most $C\lambda$.
The proof uses the following observation: if $x$ is a Lebesgue point of $A$ and $B=B(x, r)$, then
\begin{equation}\label{A-5}
\Big| A(x) -\fint_B A \Big|
\le C_d\,  r \mathcal{M} (|\nabla A|)(x).
\end{equation}

First, for  $x, y \in E_\lambda$, we use the estimate
$$
|A(x)-A(y)|
\le C_d |x-y| ( g(x) + g(y)),
$$
which follows readily from \eqref{A-5}, to obtain 
$$
|a_\lambda (x)-a_\lambda (y)|\le 2C_d \lambda |x-y|.
$$

Next, consider the case $x\in E_\lambda$ and $y \in \Omega^c$.
Let $r=2\,  \text{dist}(x,  \Omega^c )$ and $B=B(x, r)$.
Since $A=0$ in $\Omega^c$ and $|B\cap \Omega^c|\ge c |B|$, 
$$
\Big| \fint_B A \Big|\le Cr \fint_B |DA| \le C r g(x).
$$
This, together with \eqref{A-5}, gives
$$
|a_\lambda(x)-a_\lambda(y)|= |A(x)|\le C r g(x) \le C \lambda  |x-y|.
$$
As a result, we have proved that $a_\lambda$ is Lipschitz on $G_\lambda$ with constant at most $C\lambda$.

We now use Kirszbraun-Valentine extension theorem to obtain a map
$A_\lambda: \R^d \to \Lambda^{d-2} \C^d$ such that
$ A_\lambda =a_\lambda$ on $G_\lambda$ and Lip$(A_\lambda) \le C \lambda$.
Note that $A_\lambda=a_\lambda=0$ on $\Omega^c$.
Therefore, $A_\lambda\in \mathcal{V}_\infty$
and $\| D A_\lambda\|_\infty \le C \lambda$, where $C$ depends on $\Omega$.
Moreover, 
$$
|A_\lambda (x)|\le C \lambda \text{\rm dist}(x, \Omega^c) \le C \lambda.
$$
By construction, $A_\lambda=A$ a.e. on $\{ g\le \lambda \}\cap \Omega$.
Hence, up to a null set, $\{ A_\lambda \neq A \} \subset \{ g> \lambda \}$.
It follows by the locality property that 
$$
\{ DA_\lambda \neq DA \}
\subset \{ g>\lambda \},
$$
up to a null set.
Therefore,
\begin{equation}\label{A-11}
\aligned
\int_{\R^d} |D(A-A_\lambda)|^2
 & \le C \int_{ \{g> \lambda\} } (  |DA|^2 + \lambda^2)\\
 & \le C \int_{ \{ g> \lambda \} } g^2.
 \endaligned
\end{equation}

Finally, to show \eqref{A-10}, observe that
$$
\aligned
|\widetilde{T}(dA)|
& \le |\widetilde{T} (d(A-A_\lambda))|
+ |\widetilde{T} (dA_\lambda)|\\
&  \le |\widetilde{T} (d(A-A_\lambda))|
+ C M_\infty \| DA_\lambda \|_{L^\infty(\R^d)}\\
&  \le |\widetilde{T} (d(A-A_\lambda))|
+ C M_\infty \lambda.
\endaligned
$$
It follows that 
$$
\aligned
 | \left\{ |\widetilde{T}(dA)|> 2 C M_\infty \lambda\right\}|
 & \le  | \left\{ | \widetilde{T}(d (A-A_\lambda))|> CM_\infty \lambda\right \} |\\
 & \le \frac{C}{( M_\infty \lambda )^2}
 \int_{\Omega} |\widetilde{T} (d(A-A_\lambda))|^2\\
 & \le \frac{C M_2^2 }{ (M_\infty \lambda)^2}
 \int_{\R^d} |D(A-A_\lambda)|^2
\endaligned
$$
(we may assume $M_\infty\neq 0$, for otherwise $T=0$).
Therefore, for $2< p< \infty$,
$$
\aligned
\int_\Omega  |\widetilde{T}(dA)|^p
&\le C M_\infty^p  \int_0^\infty  \lambda^{p-1} | \left\{ |\widetilde{T}(dA)> 2C M_\infty \lambda \right\} |  d\lambda\\
& \le C M_\infty^{p-2} M_2^2 \int_0^\infty
\lambda^{p-3}  \left\{  \int_{\R^d} |D(A-A_\lambda)|^2\right\} d\lambda \\
&\le C M_\infty^{p-2} M_2^2
\int_0^\infty \lambda^{p-3}
\left\{ \int_{g> \lambda} g^2\right\} d\lambda \\
& \le C M_\infty^{p-2} M_2^2
\int_{\R^d}  g^p, 
\endaligned
$$
where we have used \eqref{A-11}.
Since the operator $\mathcal{M}$ is bounded on $L^p(\R^d)$ for $p>1$, we obtain 
$$
\| \widetilde{T}(dA)\|_{L^p(\Omega)}
\le C M_\infty^{1-\frac{2}{p}} M_2^{\frac{2}{p}} \| DA\|_{L^p(\R^d)}.
$$
This gives the desired estimate \eqref{A-10} and completes the proof.
\end{proof}

\begin{remark}
One may use the argument in the proof of Theorem \ref{thm-I} to prove the real
interpolation identity
\begin{equation}\label{A-1}
\Big[ L_\sigma^2 (\Omega), L^\infty_\sigma (\Omega) \Big]_{\theta, p}
=L^p_\sigma (\Omega),
\end{equation}
where $2< p< \infty$ and $\theta =1-\frac{2}{p}$.
\end{remark}

\begin{remark}
If $d=2$, $A$ is a scalar function and  the equation in \eqref{A-1a} corresponds to 
\begin{equation}\label{A-20}
u=\nabla^\perp A + h, 
\end{equation}
which has a simple proof. Indeed, 
let $\widetilde{u}$ be the zero extension of $u\in L^p_\sigma(\Omega)$.
Since div$(\widetilde{u})=0$  in $\R^2$, there exists a stream function 
$\psi\in W^{1, p}_{\loc}(\R^2, \C)$ such that 
$\widetilde{u}=\nabla^\perp \psi$.
Using $\nabla \psi=0$ in $\Omega^c$, we deduce that $\psi$ is constant in each of the 
connected components $F_0, \cdots, F_m$ of $\R^2\setminus \overline{\Omega}$,
where $F_0$ is unbounded.
By subtracting a constant, we may assume $\psi =0$ on $F_0$.
Choose smooth functions $\eta_j\in C_0^\infty(\R^2, \R)$ for $1\le j \le m$
such that $\eta_j=1$ in $F_j$ and
$\eta_j =0$ in $F_k$ for $k \neq j$.
Then there exist constants $c_1, \dots, c_m$ such that
$$
\psi =\psi_0 +\sum_{j=1}^m c_j \eta_j.
$$
and $\psi_0=0$ in $\Omega^c$.
This yields the desired representation \eqref{A-20} with 
$$
A=\psi_0 \quad \text{ and } \quad
h= \sum_{j=1}^m c_j  \nabla^\perp \eta_j.
$$
Moreover,  
$$
\aligned
\| h \|_{L^p(\R^2)}
 & \le C \sum_{j=1}^m |c_j|
\le C \| \psi \|_{L^p(\partial\Omega)}\\
& \le C \|\nabla \psi \|_{L^p(\Omega)}
= C \| u \|_{L^p(\Omega)},
\endaligned
$$
where we have used the fact $\psi =c_j$ in $F_j$.
\end{remark}


\section{Helmholtz projection and resolvent estimates in $L^p(\Omega; \C^d)$}

Let $\Omega$ be a bounded Lipschitz domain in $\R^d, d\ge 2$.
The Helmholtz projection $P: L^2(\Omega; \C^d) \to L^2_\sigma (\Omega)$ is defined by 
\begin{equation}\label{P-2}
P  f= f -\nabla \phi,
\end{equation}
where $\phi\in W^{1, 2}(\Omega; \C)$ is a weak solution of the Neumann problem,
\begin{equation}\label{N-P}
\Delta \phi =\text{\rm div}(f) \quad \text{ in } \Omega \quad \text{ and } \quad
\frac{\partial \phi}{\partial n} = n \cdot f \quad \text{ on } \partial\Omega.
\end{equation}
Let $1< p< \infty$ and $p\neq 2$. If 
\begin{equation}\label{P-p}
\sup_{{\substack{ f,  g \in C_0^\infty (\Omega; \C^d) \\  f, g \neq 0}}}
\frac{ | \langle Pf, g\rangle | }{\| f \|_{L^p(\Omega)} \| g \|_{L^{p^\prime}(\Omega)} } < \infty,
\end{equation}
 then $P$ extends  to  a bounded operator from $L^p(\Omega; \C^d)\to L^p_\sigma (\Omega)$. 
 As a result, one obtains the Helmholtz decomposition
\begin{equation}\label{HD}
L^p(\Omega; \C^d) = L^p_\sigma (\Omega) \oplus \nabla W^{1, p}(\Omega; \C).
\end{equation}
It is known that the operator $P$ is bounded on $L^p(\Omega; \C^d)$ for any $1< p< \infty$, if $\Omega$ is a bounded 
$C^1$ domain. In the case of Lipschitz domains, the operator is bounded on $L^p$ if 
\begin{equation}\label{R-L}
\Big|\frac{1}{p}-\frac12\Big|< 
\left\{
\aligned
& \frac16 + \e & \quad & \text{ if } d\ge 3, \\
& \frac14 +\e & \quad & \text{ if } d=2,
\endaligned
\right.
\end{equation} 
 where $\e>0$ depends on $\Omega$.
 Moreover, the range given by \eqref{R-L} is sharp \cite{Fabes1998}.
 
Let $F\in C_0^\infty(\Omega; \C^d)$ and  $u$ be the weak solution of \eqref{eq-0}
 in $W_0^{1, 2}(\Omega; \C^d)$.
Consider the $L^p$ estimate
\begin{equation}\label{R1}
|\lambda| \| u \|_{L^p(\Omega)}
\le C \| F \|_{L^p(\Omega)}
\end{equation}
for $1< p< \infty$.
Clearly, if $P$ is bounded on $L^p(\Omega; \C^d)$ and \eqref{R1} holds for any $F\in C_{0, \sigma}^\infty(\Omega)$,
then it holds for any $F\in C_0^\infty(\Omega; \C^d)$.
The following theorem shows that if $P$ is not bounded on $L^p(\Omega; \C^d)$,
then the $L^p$ estimate \eqref{R1} cannot hold  with a constant $C$ independent of $F\in C_0^\infty(\Omega; \C^d)$
and $\lambda>1$.
As a result, if $d\ge 3$ and $p>3$ or $p<(3/2)$, there exists a bounded Lipschitz domain $\Omega$ in $\R^d$ such that
the estimate \eqref{R1} fails to hold uniformly for $\lambda>1$.
If $d=2$, the estimate may fail for $p>4$ or $p< (4/3)$.

\begin{thm}\label{thm-B1}
Let $\Omega$ be a bounded Lipschitz domain in $\R^d, d\ge 2$ and $1< p< \infty$.
Suppose 
\begin{equation}\label{P-p1}
\sup_{{\substack{ f,g \in C_0^\infty (\Omega; \C^d) \\ f, g \neq 0}}}
\frac{ | \langle P f,  g \rangle | }{\| f \|_{L^p(\Omega)} \| g \|_{L^{p^\prime} (\Omega)}} = \infty.
\end{equation}
Then 
\begin{equation}\label{R-1}
\sup_{{\substack{ f, g \in C_0^\infty(\Omega; \C^d)\\  f, g \neq 0,  \lambda>1 }}}
\frac{ \lambda|  \langle u_\lambda, g  \rangle|  }{ \| f \|_{L^p(\Omega)} \| g\|_{L^{p^\prime}(\Omega)} }
=\infty,
\end{equation}
where $u_\lambda\in W^{1, 2}_0(\Omega; \C^d)$ is the weak solution  of \eqref{eq-0} with $F=f$ and $\lambda>1$.
\end{thm}

\begin{proof}

Let $1< p< \infty$ and $p\neq 2$.
Since $\langle Pf, g \rangle = \langle f,P g \rangle $ for any $f, g \in C_0^\infty(\Omega; \C^d)$,
 it follows  that \eqref{P-p1} holds for $p$ if and only if it holds for $p^\prime=\frac{p}{p-1}$.
The same is true for \eqref{R-1}.
As a result, we may assume $1<p< 2$ and 
\begin{equation}\label{P-p10}
\sup_{{\substack{ f \in C_0^\infty (\Omega; \C^d) \\ f \neq 0}}}
\frac{ \|  P f \|_{L^p(\Omega)} }{\| f \|_{L^p(\Omega)} } = \infty.
\end{equation}
For $k\ge 1$, 
choose $f_k \in C_0^\infty(\Omega; \C^d)$ such that $\| f_k\|_{L^p(\Omega)}=1$
and $\| Pf_k \|_{L^p(\Omega)}\ge k+1$.
Let $F_k= P f_k\in L^2_\sigma(\Omega)$ and
$u_{k, \lambda}\in W^{1, 2}_0(\Omega; \C^d)$ be the weak solution of \eqref{eq-0}
with $f_k$ in the place of $F$.
We claim that for each $k \ge 1$,  $\lambda u_{k, \lambda}  \to F_k$ in $L^2(\Omega; \C^d)$, as $\lambda \to \infty$.

Assume the claim for a moment. Since $p< 2$, we have $\lambda u_{k, \lambda} \to F_k$ in $L^p(\Omega; \C^d)$
as $\lambda\to \infty$.
Choose $\lambda_k>1 $ sufficiently large so that 
$$
\| \lambda_k u_{k, \lambda_k} - F_k \|_{L^p (\Omega)} \le 1.
$$
It follows that 
$$
\aligned
\lambda_k \| u_{k,\lambda_k} \|_{L^p(\Omega)}
& \ge \| F_k \|_{L^p(\Omega)} 
- \| \lambda_k u_{k, \lambda_k} - F_k \|_{L^p (\Omega)}\\
& \ge  (k+1)-1 =k.
\endaligned
$$
Since $\| f_k \|_{L^p(\Omega)}=1$, we obtain the equation \eqref{R-1}.

Finally, to prove the claim,
 let $F\in L^2_\sigma (\Omega)$ and $T_\lambda F =u_\lambda$ be the weak solution of
\eqref{eq-0} with $\lambda>1$.
Note that if $f\in C_{0, \sigma}^\infty(\Omega)$, 
$$
T_\lambda (-\Delta f + \lambda f ) =f.
$$
It follows that  as $\lambda \to \infty$, 
$$
\| \lambda  T_\lambda f - f \|_{L^2(\Omega)}
= \| T_\lambda (\Delta f) \|_{L^2(\Omega)} \le \lambda^{-1} \|\Delta f \|_{L^2(\Omega)} \to 0.
$$
Using $ \lambda \|T_\lambda  F  \|_{L^2(\Omega)} \le \| F  \|_{L^2(\Omega)}$ for any $F\in L^2_\sigma (\Omega)$, we obtain 
$$
\aligned
\lambda \| T_\lambda F -F \|_{L^2(\Omega)}
& \le \lambda \| T_\lambda (F-f)\|_{L^2(\Omega)}
+ \| \lambda T_\lambda f -f \|_{L^2(\Omega)}
+ \| f -  F \|_{L^2(\Omega)}\\
& \le 2 \| F-f\|_{L^2(\Omega)} + \| \lambda T_\lambda f -f \|_{L^2(\Omega)}.
\endaligned
$$
As a result, 
$$
\limsup_{\lambda\to \infty} \| \lambda  T_\lambda F -F \|_{ L^2(\Omega)}
\le 2 \| F- f \|_{L^2(\Omega)}
$$
for any $F\in L^2_\sigma (\Omega)$ and $f\in C_{0, \sigma}^\infty(\Omega)$.
This yields the claim.
\end{proof}


\section{Helmholtz projection and resolvent estimates in $L_\sigma^p(\Omega)$}\label{section-c}

In this section we show that the range $(3/2)-\e < p \le \infty$ in Theorem \ref{main-1} is sharp for 
$d=3$.

\begin{thm}\label{thm-c1}
Let $\Omega$ be a bounded Lipschitz domain in $\R^d$, $d\ge 2$.
Suppose that 
\begin{equation}\label{c-1}
L^q_\sigma (\Omega) \cap \nabla W^{1, q}(\Omega) \neq \{ 0 \}
\end{equation}
for some $q$ such that   $\frac{2d}{d+2}\le q< 2$ for $d\ge 3$ and $1< q< 2$ for $d=2$.
 Then if $1< p\le  q$, the resolvent estimate \eqref{est-0} cannot hold uniformly 
for $F\in L^p_\sigma (\Omega)$ and $\lambda>1$.
\end{thm}

\begin{proof}

We argue by contradiction. Let $1< p\le  q$.
Suppose that the resolvent estimate \eqref{est-0} holds uniformly for $F\in L^p_\sigma(\Omega)$ and 
$\lambda>1$. Let
$$
V=\{ u \in H^1_0(\Omega; \C^d): \ \text{\rm div} (u)=0 \text{ in } \Omega \}.
$$
Note that by Sobolev and Poincar\'e inequalities,
$$
\Big| \int_\Omega F \cdot \overline{\psi} \Big|
\le \| F \|_{L^q(\Omega)} \| \psi \|_{L^{q^\prime}(\Omega)}
\le C \| F \|_{L^q(\Omega)} \| \nabla \psi \|_{L^2(\Omega)}
$$
for $\psi \in V$ and $F\in L^q(\Omega; \C^d)$, where we have used the assumption on $q$.
It follows by Lax-Milgram Theorem that for any $F \in L^q(\Omega; \C^d)$ and  $\lambda>1$, there exists a unique 
$u_\lambda \in V$ such that
$$
\int_\Omega \nabla u_\lambda \cdot \nabla \overline{\psi}  +\lambda \int_\Omega u_\lambda \cdot \overline{\psi}
= \int_\Omega F \cdot \overline{\psi} 
$$
for any $\psi \in V$.
 Let $T_\lambda F = u_\lambda$. Then 
$\| T_\lambda F \|_{L^2(\Omega)} \le C \| F \|_{L^q(\Omega)}$ and
$\lambda \| T_\lambda f\|_{L^2(\Omega)} \le \| f\|_{L^2(\Omega)}$
for $f\in L^2(\Omega; \C^d)$.
Moreover, 
It follows from the proof of Theorem \ref{thm-B1} that 
\begin{equation}\label{C-1}
\| \lambda T_\lambda f - f \|_{L^2(\Omega)} 
\to 0
\end{equation} 
 as $\lambda \to \infty$, for any $ f\in L^2_\sigma(\Omega)$.

Next, let $F\in L^q_\sigma(\Omega)\cap \nabla W^{1, q}(\Omega)$ and $F\neq 0$.
Note that the condition $F\in \nabla W^{1, q}(\Omega)$ implies $\int_\Omega F \cdot \psi =0$ 
for any $\psi \in V$. It follows that $T_\lambda F=0$ for any $\lambda>1$.
Since $F \in L^q_\sigma (\Omega)$, there exists $\{ f_j \}\subset C_{0, \sigma}^\infty(\Omega)$
such that $\| f_j -F \|_{L^q(\Omega)} \to 0$.
Suppose that 
there exists $M>0$ such that 
 $\lambda \| T_\lambda f \|_{L^p(\Omega)} \le M \| f \|_{L^p(\Omega)}$ for any $f\in C_{0, \sigma}^\infty(\Omega)$
and $\lambda>1$.
Then 
$$
\lambda \| T_\lambda f_j -T_\lambda f_k \|_{L^p(\Omega)}
\le M \| f_j -f_k \|_{L^p(\Omega)}.
$$
Let $j\to \infty$. Using  $\| T_\lambda f_j \|_{L^2(\Omega)}
= \| T_\lambda (f_j -F)\|_{L^2(\Omega)} \le C \| f_j -F\|_{L^q(\Omega)}$,
we obtain 
$$
\lambda \| T_\lambda f_k \|_{L^p(\Omega)}
\le M \| F - f_k \|_{L^p(\Omega)}.
$$
In view of \eqref{C-1},  by letting  $\lambda \to \infty$,  we deduce that 
$$
\| f_k \|_{L^p(\Omega)}
\le M \| F -f_k \|_{L^p(\Omega)}.
$$
Finally, let $k \to \infty$. This yields $\| F \|_{L^p(\Omega)} =0$ and gives us a contradiction.
\end{proof}

\begin{thm}\label{thm-c2}
There exists a bounded Lipschitz domain $\Omega$, depending on $p$,  in $\R^d$ such that the resolvent estimate \eqref{est-0}
fails in $L^p_\sigma (\Omega)$ if
$$
\left\{
\aligned
d\ge 3 \text{ and } 1< p< \frac32,\\
d=2 \text{ and } 1< p< \frac43.
\endaligned
\right.
$$
\end{thm}

\begin{proof}

Consider the case $d=3$.
Fix $(6/5)< q<(3/2)$. We will construct a bounded Lipschitz domain $\Omega$ in $\R^3$ such that \eqref{c-1}
holds for $q$. By Theorem \ref{thm-c1} it follows that the resolvent estimate \eqref{est-0} fails in $L^p_\sigma(\Omega)$
for $d=3$ and $1< p\le q$. By considering the domain  $\Omega_d= \Omega \times (0, 1)^{d-3}$ in $\R^d$, 
we may also conclude that the resolvent estimate \eqref{est-0} fails in general  for $d\ge 4$ and $1< p<  (3/2)$.

Let $S=\mathbb{S}^2$ and  $\Delta_S$ denote the Laplacian on $S$.
Fix $\alpha \in (0, 1/2)$ small. Choose a smooth spherical domain $S_\alpha$ on $S$ such that  the first positive Neumann 
eigenvalue for $-\Delta_{S}$ on $S_\alpha$ is  $\alpha (\alpha +1)$.
Such domain can be constructed by connecting two regions through a narrow channel.
Let 
$$
\Omega=\Omega_\alpha =\{ r \xi : \ 0< r< 1 \text{ and } \xi \in S_\alpha \}.
$$
Let $Y\in C^\infty(\overline{S_\alpha})$ be an eigenfunction for the Neumann eigenvalue $\alpha (\alpha+1)$, normalized 
by 
$$
\int_{S_\alpha} Y  d\sigma =0  \text{ and } \| Y \|_{L^2(S_\alpha)} =1.
$$
Define
$$
a(r)=r^\alpha + \frac{\alpha}{\alpha+1} r^{-\alpha-1},  \quad h(r, \xi) = a(r) Y(\xi), \quad F=\nabla h.
$$
A computation shows that
$$
a^\prime (r)=\alpha (r^{\alpha-1} - r^{-\alpha-2}) \quad \text{ and } \quad 
(r^2 a^\prime)^\prime = \alpha (\alpha+1) a.
$$
It follows that 
$$
\Delta h=0 \quad \text{ and } \quad 
F=a^\prime(r) Y(\xi) \xi  +\frac{a(r)}{r} \nabla_S Y(\xi).
$$
Clearly, $F$ is not identically equal to $0$. We claim that 
\begin{equation}\label{c-5} 
F \in L^q_\sigma (\Omega) \cap \nabla W^{1, q} (\Omega) 
\text{ for } 1< q<q_\alpha = \frac{3}{2+\alpha}.
\end{equation}
Note that as $\alpha \to 0^+$, $q_\alpha \to (3/2)^-$.

To see \eqref{c-5}, we first observe that  $|h(r, \xi)|\le C r^{-\alpha-1}$ and
$|F(r, \xi)| \le C r^{-\alpha-2}$. It follows that 
$h \in L^2(\Omega)$ when $\alpha < (1/2)$ and
$F\in L^q(\Omega; \R^3)$ for $1< q< q_\alpha$. 
As a result, $F \in \nabla W^{1, q}(\Omega)$ for $1< q< q_\alpha$.
To show $F\in L^q_\sigma(\Omega)$, it suffices to prove that 
\begin{equation}\label{c-6}
\int_\Omega F \cdot \nabla \psi =0
\end{equation}
for any $\psi \in C_{0, \sigma}^\infty(\R^3)$.
Choose a smooth radial  function $\varphi_\e (r) $ 
such that $\varphi_\e =0$ for $r\le \e$,
$\varphi_\e =1$ for $r\ge 2\e$, and
$|\varphi_\e^\prime |\le C \e^{-1}$. 
Write
$$
\aligned
F & = (1-\varphi_\e)  F  - \frac{r \varphi_\e^\prime (r)  a^\prime (r)}{\alpha (\alpha +1)} \nabla_S Y (\xi)
+ \Big\{  \varphi_\e F + \frac{r \varphi_\e^\prime (r) a^\prime (r)}{\alpha (\alpha +1)} \nabla_S Y (\xi)\Big\}\\
&=F_1 + F_2 +F_3.
\endaligned
$$
Note that 
$$
\Big| \int_\Omega ( F_1+ F_2) \cdot \nabla \psi \Big|
\le C_\psi \int_{B(0, 2\e)} \frac{dx}{|x|^{2+\alpha}} \le C _\psi \e^{1-\alpha}.
$$
For $F_3$, we use to the formula
$$
\text{\rm div} \Big( c(r) \nabla_S Y(\xi) \Big)= \frac{c(r)}{r} \Delta_S Y(\xi)
$$
and the fact that $\text{\rm div} (F)=0$ to deduce that 
$$
\text{\rm div} (F_3)
=\nabla \varphi_\e \cdot F 
+ \frac{\varphi_\e^\prime (r) a^\prime (r) }{\alpha (\alpha +1) }\Delta_S Y =0,
$$
where we also use the fact $-\Delta_S Y = \alpha (\alpha+1) Y$.
Since $F_3 \in C^\infty(\overline{\Omega})$ and $F_3 \cdot n=0$ on $\partial\Omega$, this leads to 
$$
\int_\Omega F_3\cdot \nabla \psi =0.
$$
By letting $\e \to 0$,  we obtain \eqref{c-6} and complete the proof for $d\ge 3$.

Next, we consider the case $d=2$.
Fix $1< q< (4/3)$. 
We construct a bounded Lipschitz domain $\Omega$ in $\R^2$ such that \eqref{c-1} 
holds for $q$.  By Theorem \ref{thm-c1} the resolvent estimate \eqref{est-0} fails in $L^p_\sigma(\Omega)$
for $1< p\le q$.

Let $\omega \in (\pi, 2\pi)$ and 
$$
\Omega=\Omega_\omega = \{ (r\cos \theta, r \sin \theta)\in \R^2: \ 0< r< 1 \text{ and } 0< \theta < \omega \}.
$$
For $\alpha = \pi/ \omega\in (1/2, 1)$, define
$$
h(r, \theta)= (r^\alpha +r^{-\alpha}) \cos (\alpha \theta).
$$
Let 
$$
F=\nabla h =\alpha (r^{\alpha-1} - r^{-\alpha-1}) \cos (\alpha \theta) e_r
-\alpha (r^{\alpha-1} + r^{-\alpha-1}) \sin (\alpha \theta)  e_\theta. 
$$
Then $\text{\rm div} (F)=0$ in $\Omega$ and $F \neq 0$. 
Moreover, $F \in C^\infty(\overline{\Omega} \setminus  B(0, \e); \R^2)$ for $\e>0$ and
$F \cdot n =0$ on $\partial \Omega \setminus  \{ 0\}$. We claim that
\begin{equation}\label{c-10}
F \in L^q_\sigma (\Omega) \cap \nabla W^{1, q} (\Omega)
\text{ for } 1< q< q_\alpha = \frac{2}{1+\alpha}.
\end{equation}
As $\alpha \to (1/2)^+$, $q_\alpha \to (4/3)^-$.
This gives the desired result.

To prove \eqref{c-10}, we note that  $|h|\le C r^{-\alpha}$ and $|\nabla h |\le C r^{-\alpha-1}$. 
It follows that $F \in \nabla W^{1, q}(\Omega)$ for $1< q< \frac{2}{1+\alpha}$. 
To show $F\in L^q_\sigma(\Omega)$, we use the fact
$F=\nabla^\perp g$, where
$$
g(r, \theta) = (r^\alpha -r^{-\alpha} ) \sin (\alpha \theta).
$$
Choose  a smooth radius function $\varphi_\e$, as in the case
$d=3$. Note that  for $\psi \in C_0^\infty(\R^2; \C^2)$, 
$$
\aligned
\Big| \int_\Omega F \cdot \nabla \psi \Big|
 & \le   \Big| \int_\Omega \nabla^\perp  (\varphi_\e g)  \cdot \nabla \psi \Big|
 + \Big| \int_\Omega \nabla^\perp ( (1-\varphi_\e) g )  \cdot \nabla \psi \Big|\\
 & \le \int_\Omega |\nabla \varphi_\e | | g  | |\nabla \psi |
 + \int_\Omega |(1-\varphi_\e)| |\nabla g  |  |\nabla \psi|\\
 & \le C_\psi \e^{1-\alpha}.
\endaligned
$$
By letting $\e \to 0$, we obtain $\int_\Omega F \cdot \nabla \psi=0$.
This implies that   $F\in L^q_\sigma (\Omega)$ for $1< q< q_\alpha$,
which completes the proof.
\end{proof}

 %


 


 \bibliographystyle{amsplain}
 
 \bibliography{Shen2026-1a.bbl}

\smallskip

\begin{flushleft}

Zhongwei Shen, Institute for Theoretical Sciences, Westlake University,
No. 600 Dunyu Road, Xihu District, Hangzhou, Zhejiang, 310030, P.R. China.\\
\emph{E-mail}: \texttt{shenzhongwei@westlake.edu.cn} \\

\end{flushleft}

\bigskip

\medskip

\end{document}